\documentclass{siamltex}
\usepackage{fullpage,enumerate,setspace,changepage,multirow,rotating,xcolor}
\usepackage{graphicx,amsmath,amssymb,amsfonts,url, algorithm, algorithmic, etoolbox}
\usepackage{array,subfigure}
\usepackage[small]{caption}

\AtBeginEnvironment{align}{\setcounter{subeqn}{0}}
\newcounter{subeqn} \renewcommand{\thesubeqn}{\theequation\alph{subeqn}}%
\newcommand{\subeqn}{%
  \refstepcounter{subeqn}
  \tag{\thesubeqn}
}
\def\grad{\nabla}

\def\cA{\mathcal{A}}
\def\cB{\mathcal{B}}
\def\cC{\mathcal{C}}

\def\cF{\mathcal{F}}
\def\cG{\mathcal{G}}

\def\cI{\mathcal{I}}

\def\cK{\mathcal{K}}
\def\cL{\mathcal{L}}
\def\cM{\mathcal{M}}

\def\cO{\mathcal{O}}
\def\cP{\mathcal{P}}
\def\cQ{\mathcal{Q}}

\def\cS{\mathcal{S}}

\def\cU{\mathcal{U}}

\def\smskip{\smallskip}

\def\texitem#1{\par\smskip\noindent\hangindent 25pt
               \hbox to 25pt {\hss #1 ~}\ignorespaces}

\def\norm#1{\|#1\|}

\newcommand{\BEAS}{\begin{eqnarray*}}
\newcommand{\EEAS}{\end{eqnarray*}}
\newcommand{\BEA}{\begin{eqnarray}}
\newcommand{\EEA}{\end{eqnarray}}
\newcommand{\BEQ}{\begin{eqnarray}}
\newcommand{\EEQ}{\end{eqnarray}}
\newcommand{\BIT}{\begin{itemize}}
\newcommand{\EIT}{\end{itemize}}
\newcommand{\BNUM}{\begin{enumerate}}
\newcommand{\ENUM}{\end{enumerate}}

\newcommand{\BA}{\begin{array}}
\newcommand{\EA}{\end{array}}

\newcommand{\reals}{\mathbb{R}}
\newcommand{\integers}{\mathbb{Z}}

\newcommand{\Tr}{\mathop{\bf Tr}}

\newcommand{\argmin}{\mathop{\rm argmin}}
\newcommand{\argmax}{\mathop{\rm argmax}}

\newcommand{\dom}{\mathop{\bf dom}}

\newcommand{\relint}{\mathop{\bf rel int}}

\newif\ifpagenumbering
\pagenumberingtrue

\pagenumberingfalse

\newsavebox{\theorembox}
\newsavebox{\lemmabox}
\newsavebox{\remarkbox}
\newsavebox{\assbox}
\savebox{\theorembox}{\noindent\bf Theorem}
\savebox{\lemmabox}{\noindent\bf Lemma}
\savebox{\remarkbox}{\noindent\bf Remark}
\savebox{\assbox}{\noindent\bf Assumption}
\newtheorem{assumption}{\usebox{\assbox}}
\newtheorem{remark}{\usebox{\remarkbox}}[section]

\def\norm#1{\|#1\|}

\newcommand{\seq}{\reals^{\mathbb{N}}}
\makeatletter
\newcommand{\thickhline}{%
    \noalign {\ifnum 0=`}\fi \hrule height 1pt
    \futurelet \reserved@a \@xhline
}
\makeatother

\title{Semidefinite Programming For Chance Constrained Optimization Over Semialgebraic Sets
   \thanks{The last two authors have the same contribution.}
     }

\author{A. M. JASOUR\thanks{EE Department, The Pennsylvania State University, PA, USA ({\tt jasour@psu.edu})}
        \and N. S. AYBAT \thanks{IE Department, The Pennsylvania State University, PA, USA  } ({\tt nsa10@psu.edu})
\and C. M. LAGOA \thanks{EE Department, The Pennsylvania State University, PA, USA  ({\tt lagoa@psu.edu})} }

\begin{document}

\maketitle

\begin{abstract}
In this paper, ``chance optimization" problems are introduced, where one aims at maximizing the probability of a set defined by polynomial inequalities. These problems are, in general, nonconvex and computationally hard. With the objective of developing systematic numerical procedures to solve such problems, a sequence of convex relaxations based on the theory of measures and moments is provided, whose sequence of optimal values is shown to converge to the optimal value of the original problem. Indeed, we provide a sequence of semidefinite programs of increasing dimension which can arbitrarily approximate the solution of the original problem. To be able to efficiently solve the resulting large-scale semidefinite relaxations, a first-order augmented Lagrangian algorithm is implemented.
Numerical examples are presented to illustrate the computational performance of the proposed approach.
\end{abstract}




\section{Introduction} \label{intro}
In this paper, we aim at solving \emph{chance optimization problems;} i.e., problems which involve maximization of the probability of a semialgebraic set defined by polynomial inequalities. More precisely, given a probability space $\left(\reals^m,\bar{\Sigma}_q,\bar{\mu}_q\right)$ with $\bar{\Sigma}_q$ denoting the Borel $\sigma$-algebra of $\reals^m$ and $\bar{\mu}_q:\bar{\Sigma}_q\rightarrow\reals_+$ denoting a finite (positive) Borel measure on $\bar{\Sigma}_q$, we focus on the problem given in \eqref{intro_union} over decision variable $x \in {{\mathbb{R} }}^n$.
\begin{align} \label{intro_union}
\mathbf{P^*} := \sup_{x\in\reals^n}~\bar{\mu}_q \left(\bigcup_{k=1,\ldots,N}\bigcap_{j=1,\ldots,\ell_k} \left\{ q\in\reals^m:\  \mathcal{P}^{(k)}_j(x,q)\geq0 \right\}\right),
\end{align}
where 
$\mathcal{P}^{(k)}_j:{{\mathbb{R}}}^n\times {{\mathbb{R}}}^m\rightarrow{\mathbb{R}}$, $j=1,2,\dots ,\ell_k$ and $k=1,\ldots,N$ are given polynomials. Let $\mathcal K_k :=\left\{(x,q): \mathcal{P}_j^{(k)}(x,q)\geq 0,\ j=1,\ldots,\ell_k\right\}$ and $\cK:=\bigcup_{k=1}^N\cK_k$. Under the assumption that $\cK$ is bounded, we show that by solving a sequence of semidefine programming~(SDP) problems of growing dimension, we can construct a sequence $\{\mathbf{y}^d_\mathbf{x}\}_{d\in\integers_+}\subset\seq$ that has an accumulation point in the weak-$\star$ topology of $\ell_\infty$, and for every accumulation point $\mathbf{y_x^*}\in\seq$, there is a representing finite (positive) Borel measure $\mu_x^*$ such that any $x^*\in supp(\mathbf{\mu_x^*})$ is an optimal solution to \eqref{intro_union}, i.e., the supremum $\mathbf{P^*}$ is attained at $x^*$, where $\seq$ denotes the vector space of real sequences. Note that the problem of interest in \eqref{intro_union}, when reformulated in \emph{hypograph} form, can be equivalently written as a chance constrained optimization problem: 
$\sup_{x\in R^n, \gamma\in\reals}\left\{\gamma:\ \bar{\mu}_q \left(\bigcup_{k=1,\ldots,N}\bigcap_{j=1,\ldots,\ell_k} \left\{ q\in \reals^m:\ \mathcal{P}^{(k)}_j(x,q)\geq0 \right\}\right) \geq \gamma\right\}$. First, the emphasis will be placed on the following special case of \eqref{intro_union}, where $N=1$,
\begin{equation}
\label{intro_max}
\mathbf{P^*} :=\sup_{x\in\reals^n}~\bar{\mu}_q \bigg(\big\{ q\in\reals^m:\  \mathcal{P}_j(x,q)\geq 0,\quad j=1,\ldots,\ell \big\}\bigg),
\end{equation}
and then all the results derived for the special case \eqref{intro_max} will be extended to the case where $N>1$.

The potential application area of this problem class is quite large and encompasses many well-known problems in different areas as special cases. For example, designing probabilistic robust controllers~\cite{r_App_Control1}, model predictive controllers in presence of random disturbances \cite{r_App_Predictive Control2,r_App_Predicitive Control1,r_App_Predictive Control3}, and optimal path planning and obstacle avoidance problems in robotics \cite{r_App_Robotic 1,r_App_Robotic 3,r_App_Robotic 2} can be cast as special cases of this framework. Moreover, problems in the area of economics, finance, and trust design~\cite{r_App_Economy 3,r_App_Economy 1,r_App_Economy 2} can also be formulated as \eqref{intro_union} and \eqref{intro_max}. Although, in some particular cases, the problem in \eqref{intro_union} is convex (e.g., see \cite{r_convex_chance 1,r_convex_chance 2}), in general, chance constrained problems are not convex; e.g., see \cite{r_convex_chance 1} for non-convex chance constrained linear programs. In this paper, we use previous results on moments of measures (e.g., see \cite{r_measure_ref 1,r_measure_ref 2}) to develop a sequence of SDP problems, known as Lasserre's hierarchy~\cite{r_measure_ref 2}, whose solutions converge to the solution of \eqref{intro_union}.


\subsection{ Previous Work}

Several approaches have been proposed to solve chance constrained problems. The main idea behind most of the proposed methods is to find a tractable approximation for chance constraints. One particular method is the so-called \emph{scenario approach;} see \cite{r_scenario 3,r_scenario 4,r_scenario 1,r_scenario 2,r_scenario 5} and the references therein. In this approach, the probabilistic constraint is replaced by a (large) number of deterministic constraints obtained by drawing independent identically distributed~(iid) samples of random parameters. Being a randomized approach, there is always a positive probability of failure (perhaps small). In \cite{r_robust 1,r_robust 2,r_robust 3,r_robust 4,r_robust 5}, robust optimization is used to deal with uncertain linear programs (LP). In this method, the uncertain LP is replaced by its robust counterpart, where the worst case realization of uncertain data is considered. The proposed method is not computationally tractable for every type of uncertainty set. A specific case that is tractable is LP with ellipsoidal uncertainty set \cite{r_robust 1}. In \cite{r_upper 1,r_upper 3,r_upper 2}, an alternative approach is proposed where one analytically determines an upper bound on the probability of constraint violation. Although this method does provide a convex approximation, it can only be applied to specific uncertainty structures. In \cite{r_bern 1,r_bern 2} the authors propose the so-called Bernstein approximation where a convex conservative approximation of chance constraints is constructed using 
generating functions. Although approximation is efficiently computable, it is only applicable to problems with convex constraints that are affine in random vector $q\in {{\mathbb{R} }}^m$. Moreover, components of $q$ need to be independent and have computable finite generating functions. 
In \cite{r_kin 1,r_kin 2,r_kin 3} convex relaxations of chance constrained problems are presented. The concept of polynomial kinship function is used to estimate an upper bound on the probability of constraint violation. Solutions to a sequence of relaxed problems are shown to converge to a solution of the original problem as the degree of the polynomial kinship function increases along the sequence. In \cite{r_kin 3,r_indi}, an equivalent convex formulation is provided based on the theory of moments. In this method the probability of a polynomial being negative is approximated by computing polynomial approximations for univariate indicator functions \cite{r_indi}.

\emph{Distributionally robust} chance constrained programming -- see \cite{r_ref2 1, r_ref2 2, r_ref2 3, r_ref2 4, r_ref2 5}, is another popular tool for dealing with uncertainty in the problem, where only a finite number of moments $m_\alpha$
of the underlying measure $\bar{\mu}_q$ are assumed to be known, i.e., $\{m_\alpha\}_{\alpha\in \cA}$ is known for $\cA\subset\mathbb{N}^m$ such that $|\cA|<\infty$. In this approach robust chance constraints are formulated by considering the worst case measure within a family of measures with moments equal to $\{m_\alpha\}_{\alpha\in A}$. However, proposed methods in this literature are mainly limited to linear chance constraints and/or to specific types of uncertainty distributions. For instance, in~\cite{r_ref2 1}, under the assumption $\bar{m}=E_{\bar{\mu}_q}[q]$ and $\bar{S}=E_{\bar{\mu}_q}[(q-\bar{m})(q-\bar{m})^T]$ are known, the linear chance constraint of the form $\bar{\mu}_q\left(\{q:\ q^T x\geq 0\}\right)\geq 1-\epsilon$ is replaced by its robust counterpart: $\inf_{\mu_q\in\cM}\mu_q\left(\{q:\ q^T x\geq 0\}\right)\geq 1-\epsilon$, where $\cM$ is the set of finite (positive) Borel measures on $\bar{\Sigma}_q$ with their means and covariances equal to $\bar{m}$ and $\bar{S}$, respectively; and it is shown that these robust constraints can be represented as second-order cone constraints for a wide class of probability distributions. 
In~\cite{r_ref2 2}, the authors has reviewed and developed different approximation methods for problems with joint chance constraints. 
In the proposed method, joint chance constraints are decomposed into individual chance constraints, and classical robust optimization approximation is used to deal with the new constraints. In \cite{r_ref2 3} a tractable approximation method for probabilistically dependent linear chance constraints is presented. In \cite{r_ref2 4} linear chance constraints with Gaussian and log-concave uncertainties are addressed, and it is shown that they can be reformulated as semi-infinite optimization problems; moreover, tight probabilistic bounds are provided for the resulting \emph{comprehensive robust optimization} problems~\cite{BenTal10,BenTal06}. In~\cite{r_ref2 5} an SDP formulation is provided to approximate distributionally robust chance constraints where only the support of $\bar{\mu}_q$, and its first and second order moments are known.

In this paper, we take a different approach to deal with chance constrained problems. The proposed method is based on volume approximation results in \cite{r_vol} and the theory of moments \cite{r_measure_ref 1,r_measure_ref 2}. In \cite{r_vol}, a hierarchy of SDP problems are proposed to compute the volume of a given compact semialgebraic set. 
It is shown 
that the volume of a semialgebraic can be computed by solving a maximization problem over finite (positive) Borel measures supported on the given set, and restricted by the Lebesgue measure on a simple set containing the semialgebraic set of interest. Building on this result, we propose the \emph{chance optimization} problem over semialgebraic sets --see our preliminary results in~\cite{r_our work}. In particular, we address the problem of probability maximization over the union of semialgebraic sets defined by intersections of finite number of polynomial inequalities as in~\eqref{intro_union}. Here, one needs to search for the (positive) Borel measure with maximum possible mass on the given semialgebraic set, while simultaneously searching for an upper bound probability measure over a simple set containing the semialgebraic set and restricting the Borel measure.
%
\subsection{The Sequel}
The outline of the paper is as follows: in Section 2, the notation adopted in the paper, and preliminary results on measure theory are presented; in Sections~3 and 4, we propose equivalent problems, and sequences of SDP relaxations to \eqref{intro_max} and \eqref{intro_union}, respectively; and show that the sequence of optimal solutions to SDP relaxations converge to the solutions of the original problems. 
In Section~5, we implement an efficient first-order algorithm to solve regularized SDP relaxations of the chance constrained problems, and finally, present numerical results, followed by some concluding remarks given in Section~6. 


\section{ Notation and Preliminary Results}

\subsection{ Notations and Definitions}
\label{sec:definitions}
Throughout the paper, given a sequence $\mathbf{p}=\{p_\alpha\}_{\alpha\in\cA}\subset\reals$ over a countable index set $\cA\subset\mathbb{N}^n$, we assume that the elements of $\cA$ is sorted according to \emph{graded reverse lexicographic order} (grevlex): $\cA=\{\alpha^{(i)}:\ i=1,\ldots,|\cA|\}$ such that $\alpha ^{(1)} <_g \ldots <_g \alpha ^{(|\cA|)}$, where $|\cA|$ denotes the cardinality of $\cA$; and the order on $\cA$ also induces an order on the elements of $\mathbf{p}=[p_{\alpha^{(1)}},\ldots, p_{\alpha^{(|\cA|)}}]^T\in\reals^{|\cA|}$. Throughout the paper the notation $(\bf{p})_\alpha$ refers to $p_\alpha$. Let $\mathbb{R}[x]$ be the ring of real polynomials in the variables $x \in \mathbb{R}^n$. Given $\cP\in\mathbb{R}[x]$, we will represent $\cP$ as $\sum_{\alpha\in\mathbb{N}^n} p_\alpha x^\alpha$ using the standard basis $\{x^\alpha\}_{\alpha\in \mathbb{N}^n}$ of $\mathbb{R}[x]$, where $x^\alpha:=\Pi_{j=1}^nx_j^{\alpha_j}$, and $\mathbf{p}=\{p_\alpha\}_{\alpha\in\mathbb{N}^n}$ denotes the sequence of polynomial coefficients. Note that $\mathbf{p}$ contains finitely many nonzeros, and we assume that the elements of the coefficient sequence $\mathbf{p}=\{p_\alpha\}_{\alpha\in\mathbb{N}^n}$ are sorted according to \emph{grevlex} order on the corresponding monomial exponent $\alpha$. 
 Given $\mathbf{y}=\{y_\alpha\}_{\alpha\in\mathbb{N}^n}\subset\reals$, 
 let $L_\mathbf{y}:\reals[x]\rightarrow\reals$ be a linear map defined as
\begin{equation}
\label{eq:lin_map}
\cP \quad \mapsto \quad L_\mathbf{y}(\cP)=\sum_{\alpha\in\mathbb{N}^n}p_\alpha y_\alpha, \quad \hbox{where} \quad \cP(x)=\sum_{\alpha\in\mathbb{N}^n} p_\alpha x^\alpha.
\end{equation}
Given $n$ and $d$ in $\mathbb{N}$, we define $S_{n,d} := \binom{d+n}{n}$ and $\mathbb{N} ^{\rm n}_d := \{\alpha \in \mathbb N^n : \norm{\alpha}_1 \leq d \}$. Let $\mathbb R_{\rm d}[x] \subset \mathbb R [x]$ denote the set of polynomials of degree at most $d\in \mathbb{N}$, which is indeed a vector space of dimension $S_{n,d}$. Similar to $\cP\in\mathbb{R}[x]$, given $\cP\in\mathbb R_{\rm d}[x]$, $\mathbf{p}=\{p_\alpha\}_{\alpha\in\mathbb{N}^{\rm n}_d}$ is sorted such that $\mathbf{p}=[p_{\alpha^{(1)}},\ldots, p_{\alpha^{(S_{n,d})}}]^T\in\reals^{S_{n,d}}$, where $\mathbb{N} ^{\rm n}_d\ni\mathbf{0} = \alpha ^{(1)} <_g \ldots <_g \alpha ^{(S_{n,d})}$. 
Moreover, let $\mathbb{S}^2[x] \subset \mathbb R [x]$ be the set of sum of squares~(SOS) polynomials. $s:\reals^n\rightarrow\reals$ is an SOS polynomial if it can be written as a sum of \emph{finitely} many squared polynomials, i.e., $s(x)= \sum_{j=1}^{\ell} h_j(x)^2$ for some $\ell<\infty$ and $h_j\in\reals[x]$ for $1\leq j\leq \ell$.

Let $\seq$ denote the vector space of real sequences, and let $\cM(\cK)$ be the set of finite (positive) Borel measures $\mu$ such that $supp(\mu)\subset\cK$, where $supp(\mu)$ denotes the support of the measure $\mu$; i.e., the smallest closed set that contains all measurable sets with strictly positive $\mu$ measure. A sequence $\mathbf y = \{ y_ \alpha \}_{\alpha\in\mathbb{N}^n}\in\seq$
is said to have a \emph{representing measure}, if there exists a finite Borel measure $\mu$ on $\reals^n$ such that $y_{\alpha } = \int{x^{\alpha} d\mu }$ for every $\alpha \in \mathbb N ^n$ -- see~\cite{r_measure_ref 1, r_measure_ref 2}. In this case, $ \mathbf y $ is called the moment sequence of the measure $\mu $. Given two measures $\mu_1$ and $\mu_2$ on a Borel $\sigma $-algebra $\Sigma $, the notation $\mu_1 \preccurlyeq \mu _2$ means  $\mu _1(S)\le \mu_2(S)$ for any set $S \in \Sigma $. Moreover, if $ \mu_1$ and $ \mu_2$ are both measures on Borel $\sigma$-algebras $\Sigma_1$ and $\Sigma_2$, respectively, then $\mu =\mu_1 \times \mu_2$ denotes the product measure satisfying $ \mu(S_1 \times S_2)=\mu_1 (S_1) \mu_2(S_2)$ for any measurable sets $S_1\in \Sigma_1$, $S_2 \in \Sigma_2$ \cite{r_vol}. Let $C\subset\reals^n$, $\Sigma(C)$ denotes the Borel $\sigma$-algebra over $C$. Given two square symmetric matrices $A$ and $B$, the notation $A \succcurlyeq 0$ denotes that $A$ is positive semidefinite, and $A\succcurlyeq B$ stands for $A-B$ being positive semidefinite.

\textbf{Putinar's property:} A closed semialgebraic set $\cK = \{ x\in \mathbb{R}^n: \mathcal{P}_j(x)\geq0,\ j=1,2,\dots ,\ell\ \}$ defined by polynomials $\mathcal{P}_j\in \mathbb R [x]$ satisfies \emph{Putinar's property}~\cite{r_measure_ref 6} if there exists $\mathcal{U}\in \mathbb R [x]$ such that $\lbrace x:  \mathcal{U}(x) \geq 0 \rbrace $ is compact and $\mathcal{U} = s_0 + \sum_{j=1}^{\ell} s_j\mathcal{P}_j $ for some SOS polynomials $\{s_j\}_{j=0}^\ell \subset \mathbb{S}^2[x]$ -- see~ \cite{r_measure_ref 5,r_measure_ref 2,r_measure_ref 6}. Putinar's property holds if the level set $\lbrace x: \mathcal{P}_j(x) \geq 0 \rbrace$ is compact for some $j$, or if all $ \mathcal{P}_j $
are affine and $\cK $ is compact - see~\cite{r_measure_ref 5}. 
Putinar's property is not a geometric property of the semi-algebraic set $\cK$, but rather an algebraic property related to the representation of the set by its defining polynomials. Hence, if there exits $M>0$ such that the polynomial $\cP_{\ell+1}(x):= M- \Vert x \Vert^2 \geq 0$ for all $x\in\cK$, then the \textit{new representation} of the set $\cK = \{ x\in \mathbb{R}^n: \mathcal{P}_j(x)\geq0,\ j=1,2,\dots ,\ell+1\ \}$ satisfies Putinar's property.


\textbf{Moment matrix:} Given $d\geq 1$ and a sequence $\{y_\alpha\}_{\alpha\in\mathbb{N}^n}$, the moment matrix $M_d({\mathbf y})\in\reals^{S_{n,d}\times S_{n,d}}$ is a symmetric matrix and its $(i,j)$-th entry is defined as follows~\cite{r_measure_ref 1,r_measure_ref 2}:
\begin{equation}\label{momnt matirx def}
M_d ( \mathbf y )(i,j):= L_{\mathbf y}\left(x^{\alpha^{(i)}+\alpha^{(j)}}\right)=y_{\alpha^{(i)}+\alpha^{(j)}},\ \ \ 1 \leq i,j \leq S_{n,d},
\end{equation}
where 
$\mathbb{N} ^{\rm n}_d=\{\alpha ^{(i)}\}_{i=1}^{S_{n,d}}$ such that $\mathbf{0} = \alpha ^{(1)} <_g \ldots <_g \alpha ^{(S_{n,d})}$ are sorted according to \emph{grevlex} order. 

Let $\cB_d^T=\left[x^{\alpha^{(1)}},\ldots, x^{\alpha^{(S_{n,d})}}\right]^T$ denote the vector comprised of the monomial basis of $\mathbb R_{\rm d}[x]$. Note that the moment matrix can be written as $M_d({\mathbf y}) = L_\mathbf{y}\left(\cB_d \cB_d^T\right)$; here, the linear map $L_\mathbf{y}$ operates componentwise on the matrix of polynomials, $\cB_d \cB_d^T$. For instance, let $d=2$ and $n=2$; the moment matrix containing moments up to order $2d$ is given as
\begin{equation} \label{moment matrix exa}
M_2\left({\mathbf y}\right)=\left[ \begin{array}{c}

\begin{array}{ccc} y_{00} \ | & y_{10} & y_{01}| \end{array}
\begin{array}{ccc} y_{20} & y_{11} & y_{02} \end{array}
 \\
 \begin{array}{ccc} - & - & - \end{array}
\ \ \ \  \begin{array}{ccc} - & - & - \end{array}
 \\

 \begin{array}{ccc} y_{10}\ | & y_{20} & y_{11}| \end{array}
\ \begin{array}{ccc} y_{30} & y_{21} & y_{12} \end{array}
 \\

 \begin{array}{ccc} y_{01}\ | & y_{11} & y_{02}| \end{array}
\ \begin{array}{ccc} y_{21} & y_{12} & y_{03} \end{array}
 \\

 \begin{array}{ccc} - & - & - \end{array}
\ \ \ \ \  \begin{array}{ccc} - & - & - \end{array}
 \\

 \begin{array}{ccc} y_{20}\ | & y_{30} & y_{21}| \end{array}
\ \begin{array}{ccc} y_{40} & y_{31} & y_{22} \end{array}
 \\

 \begin{array}{ccc} y_{11}\ | & y_{21} & y_{12}| \end{array}
\ \begin{array}{ccc} y_{31} & y_{22} & y_{13} \end{array}
 \\

 \begin{array}{ccc}y_{02}\ | & y_{12} & y_{03}| \end{array}
\ \begin{array}{ccc} y_{22} & y_{13} & y_{04} \end{array}

 \end{array}
\right].
\end{equation}
\textbf{Localizing matrix:} Given a polynomial $\mathcal{P} \in \mathbb R [x]$, let $ \mathbf p = \{ p_{\gamma }\}_{\gamma\in\mathbb{N}^n}$ be its coefficient sequence in standard monomial basis, i.e., $\cP(x)=\sum_{\alpha\in\mathbb{N}^n} p_\alpha x^\alpha$, 
the $(i,j)$-th entry of the \emph{localizing matrix} $M_d(\mathbf{y};\mathbf{p})\in\reals^{S_{n,d}\times S_{n,d}}$ with respect to $\mathbf y $ and $\mathbf p$ is defined as follows~\cite{r_measure_ref 1,r_measure_ref 2}:
\begin{equation}\label{localization matrix def}
M_d(\mathbf y;\mathbf p)(i,j) := L_{\mathbf{y}}\left(\cP x^{\alpha^{(i)}+\alpha^{(j)}}\right)=\sum_{\gamma \in \mathbb N^n} p_{\gamma} y_{\gamma +\alpha^{(i)}+\alpha^{(j)}}, \ \ 1 \leq i,j \leq  S_{n,d}.
\end{equation}
Equivalently, $M_d(\mathbf y;\mathbf p) = L_{\mathbf{y}}\left(\mathbf \cP \cB_d\cB_d^T\right)$, where $L_{\bf y}$ operates componentwise on $\cP \cB_d\cB_d^T$. For example, given $\mathbf{y}=\{y_\alpha\}_{\alpha\in\mathbb{N}^2}$ and the coefficient sequence $\mathbf{p}=\{p_\alpha\}_{\alpha\in\mathbb{N}^2}$ corresponding to polynomial $\cP$,
\begin{equation}
 \mathcal{P}(x_1,x_2)=a-bx_1-cx^2_2,
\end{equation}
the localizing matrix for $d=1$ is formed as follows
\begin{equation}
  M_1(\mathbf{y};\mathbf p)= \begin{small}
 \left[ \begin{array}{ccc}
{ay}_{00}-by_{10}-cy_{02} & {ay}_{10}-by_{20}-cy_{12} & {ay}_{01}-by_{11}-cy_{03} \\
{ay}_{10}-by_{20}-cy_{12} & {ay}_{20}-by_{30}-cy_{22} & {ay}_{11}-by_{21}-cy_{13} \\
{ay}_{01}-by_{11}-cy_{03} & {ay}_{11}-by_{21}-cy_{13} & {ay}_{02}-by_{12}-cy_{04} \end{array}
\right].
 \end{small}
\end{equation}
\subsection{ Preliminary Results}

%
In this section, we state some standard results found in the literature that will be referred to later in Sections~\ref{sec:intersection_prob} and~\ref{sec:union_prob}.
\begin{lemma}\label{lemma1}
Let $\mu$ be a Borel probability measure supported on the hyper-cube $[-1, 1]^n$. Its moment sequence $\mathbf y\in\reals^\mathbb{N}$ satisfies $\Vert \mathbf y \Vert_{\infty} \leq 1$.
\end{lemma}
\begin{proof}
Since $supp(\mu) \subset [-1, 1]^n$ and $\mu$ is a probability measure, we have
$\vert y_{\alpha } \vert \leq \int{ \vert x^{\alpha} \vert d\mu }\leq \int{ \vert x \vert d\mu }\leq 1$ for each $\alpha\in \mathbb N^n$. Hence, $\Vert \mathbf y \Vert_{\infty} \leq 1 $.  
\end{proof}

The following lemmas give necessary, and sufficient conditions for $\mathbf y$ to have a representing measure $\mu$ -- for details see \cite{r_vol,r_measure_ref 4,r_measure_ref 2}.
\begin{lemma}
\label{lem:nec_cond}
Let $\mu$ be a finite Borel measure on $\reals^n$, and $\mathbf{y}=\{y_\alpha\}_{\alpha\in\mathbb{N}^n}$ such that $y_\alpha=\int x^\alpha d\mu$ for all $\alpha\in\mathbb{N}^n$. Then $M_d(\mathbf y)\succcurlyeq 0$ for all $ d\in \mathbb N$.
\end{lemma}

\begin{lemma}
\label{mom_bound 1}
Let $\mathbf y = \{y_{\alpha}\}_{\alpha\in\mathbb{N}^n}$ be a real sequence. If $M_d({\bf y}) \succcurlyeq 0$ for some $d\geq 1$, then
\begin{equation*}
\vert y_{\alpha} \vert \leq \max \left\{ y_0, \max_{i=1,\ldots, n} L_{\mathbf{y}}\left( x_i^{2d}\right) \right\}\quad \forall \alpha\in\mathbb{N}^n_{2d}.
\end{equation*}
\end{lemma}
\begin{lemma}
\label{lem:suf_cond_prob}
If there exist a constant $c > 0$ such that $M_d(\mathbf y)\succcurlyeq 0$ and $|y_{\alpha}| \leq c$ for all $ d\in \mathbb N$ and $\alpha\in\mathbb{N}^n$, then there exists a representing measure $\mu$ with support on $[-1, 1]^n$.
\end{lemma}

Given polynomials $\mathcal{P}_j\in \mathbb R [x]$, let $\textbf{p}_j$ be its coefficient sequence in standard monomial basis for $j=1,2,\dots ,\ell$; consider the semialgebraic set $\cK$ defined as
\begin{equation}\label{preliminary result_semi algebraic set}
\cK = \{ x\in \mathbb{R}^n: \mathcal{P}_j(x)\geq0,\ j=1,2,\dots ,\ell\ \}.
\end{equation}
The following lemma gives a necessary and sufficient condition for $\mathbf y$ to have a representing measure $\mu$ supported on $\cK$ -- see \cite{r_vol,r_measure_ref 4,r_measure_ref 1,r_measure_ref 2}.
\begin{lemma}
\label{preliminary result_measure 2}
If $\cK$ defined in \eqref{preliminary result_semi algebraic set} satisfies Putinar's property, 
then the sequence $\mathbf y = \{ y_\alpha\}_{\alpha\in\mathbb{N}^n}$ has a \emph{representing} finite Borel measure $\mu$ on the set $\cK$, if and only if
\begin{equation*}
M_d(\mathbf y)\succcurlyeq 0,\quad M_d(\mathbf y;\textbf{p}_j)\succcurlyeq 0,\ \ j=1,\dots ,\ell, \hbox{ for all }  d\in \mathbb N.
\end{equation*}
\end{lemma}
Finally, the following lemma, proven in~\cite{r_vol}, shows that the Borel measure of a compact set is equal to the optimal value of an infinite dimensional LP problem.
\begin{lemma}
\label{preliminary result_volume}
Let $\Sigma$ be the Borel $\sigma$-algebra on $\reals^n$, and $ \mu_1$ be a measure on a compact set $\cB\in\Sigma$. Then for any given $\cK\in\Sigma$ such that $\cK\subseteq  \cB$, one has
\begin{equation*}
\mu_1(\cK)= \int_{\cK} d\mu_1 = \sup_{\mu_2\in\cM(\cK)} \left\lbrace  \int d\mu_2 : \mu_2 \preccurlyeq \mu_1\right\rbrace,
\end{equation*}
where $\cM(\cK)$ is the set of finite Borel measures on $\cK$.
\end{lemma}

\section{Chance Optimization over a Semialgebraic Set}
\label{sec:intersection_prob}
In this section we focus on the \emph{chance optimization} problem stated in~\eqref{intro_max}. We first provide an equivalent problem over finite (positive) Borel measures as variables, and then we will consider its relaxations in the moment space. 
Given polynomials $\mathcal{P}_j:\reals^n\times\reals^m\rightarrow\reals$ with degree $\delta_j$ for $j=1,\ldots,\ell$,  we define
\begin{equation}\label{max_p2_set}
\mathcal K =\{(x,q)\in\reals^n\times\reals^m:\ \mathcal{P}_j(x,q) \geq 0,\ j=1,2,\dots,\ell\}.
\end{equation}

\begin{assumption}
\label{assump:putinar}
$\mathcal K$ satisfies Putinar's property.
\end{assumption}

\begin{remark}
\label{rem:prob_space}
Assumption~\ref{assump:putinar} implies that $\cK$ is a compact set; hence the projections of $\cK$ onto $x$-coordinates and onto $q$-coordinates, i.e., $\Pi_1=:\{x\in\reals^n:\ \exists q\in\reals^m \hbox{ s.t. } (x,q)\in\cK\}$ and $\Pi_2=:\{q\in\reals^m:\ \exists x\in\reals^n \hbox{ s.t. } (x,q)\in\cK\}$, are also compact. Therefore, after rescaling of polynomials, we assume without loss of generality that $\Pi_1\subset\chi:=[-1,1]^n$ and $\Pi_2\subset\cQ:=[-1,1]^m$. Furthermore, instead of working on the original probability space $(\reals^m,\bar{\Sigma}_q,\bar{\mu}_q)$, we can adopt a smaller probability space $(\cQ,\Sigma_q,\mu_q)$, where $\Sigma_q:=\{S\cap\cQ:\ S\in\bar{\Sigma}_q\}$ and $\mu_q(S):=\frac{\bar{\mu}_q(S)}{\bar{\mu}_q(\cQ)}$ for all $S\in\Sigma_q$. Therefore, we can take for granted that $\mu_q\in\cM(\cQ)$,
where $\cM(\cQ)$ is the set of finite Borel measures $\mu_q$ such that $supp(\mu_q) \subset \mathcal{Q}$.
We also assume that moments of any order of $\mu_q$ can be computed.
\end{remark}
\subsection{An Equivalent Problem}
As an intermediate step in the development of convex relaxations of the original problem, a related infinite dimensional problem in the measure space is provided below: 
\begin{align}
\mathbf{P_{\mu_q}^*}:=&\ \sup_{\mu ,\mu_x} \int d\mu, \label{max_p2}\\
\hbox{s.t.}\quad & \mu \preccurlyeq \mu_x \times \mu_q, \label{eq:mesuare_ineq}\subeqn\\
&\mu_x \hbox{ is a probability measure} \label{eq:prob_measure_const}, \subeqn\\
&\mu_x\in \cM(\chi),\quad \mu\in\cM(\mathcal K).  \label{eq:support_const}\subeqn
\end{align}
\vspace{-0.3cm}
\begin{theorem} \label{Max Theo 1}
The optimization problems in \eqref{intro_max} and \eqref{max_p2} are equivalent in the following sense:
\begin{enumerate}[i)]
\item The optimal values are the same, i.e., $\mathbf{P^*}=\mathbf{P_{\mu_q}^*}$.
\item If an optimal solution to \eqref{max_p2} exists, call it $\mu^*_x$, then any $x^*\in supp(\mu^*_x)$ is an optimal solution to \eqref{intro_max}.
\item If an optimal solution to \eqref{intro_max} exists, call it $x^*$, then $\mu_x = \delta_{x^*}$, Dirac measure at $x^*$, and $\mu = \delta_{x^*} \times \mu_q$ is an optimal solution to \eqref{max_p2}.
\end{enumerate}
\end{theorem}
\begin{proof}
Let $(\cQ,\Sigma,\mu_q)$ be the probability space defined in Remark~\ref{rem:prob_space}.
Note that since $\mathcal{P}_j(x,q)$ is a polynomial in random vector $q\in\reals^m$ for all $x\in\reals^n$, it is continuous in $q$; hence $\mathcal{P}_j(x,.)$ is Borel measurable for all $x\in\reals^n$ and $j=1,\ldots,\ell$. As discussed in Remark~\ref{rem:prob_space}, it can be assumed that $\cK\subset\chi\times\cQ=[-1,1]^n\times[-1,1]^m$. Define $\mathcal F:\mathbb R^n\rightarrow\Sigma$ as follows
\begin{equation}\label{eq:F_func}
\mathcal F(x) := \{q\in\reals^m:\mathcal{P}_j(x,q)\geq0,\ j=1,2,\dots ,\ell\},
\end{equation}
and consider the following problem over the probability measures in $\cM(\chi)$:
\begin{equation}\label{eq:aux_problem}
\mathbf{P}:=\sup_{\mu_x\in\cM(\chi)}\left\{ \int_\chi \mu_q (\mathcal F(x))~d \mu_x:\ \mu_x(\chi)=1\right\}.
\end{equation}

Note that the optimal value of \eqref{intro_max} can be written as $\mathbf{P^*}=\sup_{x\in\chi}\mu_q(\cF(x))$. Let $\mu_x$ be a feasible solution to \eqref{eq:aux_problem}. Since $\mu_q(\mathcal F(x)) \le \mathbf{P^*}$ for all $x \in \chi$, we have $\int \mu_q (\mathcal F(x))~d\mu_x \le \mathbf{P^*}$. Thus, $\mathbf{P} \le \mathbf{P^*}$. Conversely, let $x\in\reals^n$ be a feasible solution to the problem in \eqref{intro_max} and $\delta_{x}$ denote the Dirac measure at $x$. The objective value of $x$ in \eqref{intro_max} is equal to $\mu_q (\mathcal F(x))$. 
Moreover, $\mu_x = \delta_{x}$ is a feasible solution to the problem in \eqref{max_p2} with objective value equal to $\mu_q (\mathcal F(x))$. This implies that $\mathbf{P^*} \le \mathbf{P}$. Hence, $\mathbf{P^*} = \mathbf{P}$, and 
\eqref{eq:aux_problem} can be rewritten as
\begin{equation} \label{max_p2_proof_p2_2}
\mathbf{P^*}=\sup_{\mu_x\in\cM(\chi)}\left\{\int_{\chi} \int_{\mathcal{F}(x)} d\mu_q d\mu_x:\ \mu_x(\chi)=1\right\} = \sup_{\mu_x\in\cM(\chi)}\left\{\int_{\mathcal K} d\mu_x \mu_q:\ \mu_x(\chi)=1\right\},
\end{equation}
and using the epigraph formulation shown in Lemma~\ref{preliminary result_volume}, we finally obtain 
\begin{equation*} \label{max_p2_proof_p2_3}
\mathbf{P^*}=\sup_{\mu_x\in\cM(\chi)} \sup_{\mu\in\cM(\cK)} \int d\mu \quad \hbox{s.t.}\quad \mu \preccurlyeq \mu_x \times \mu_q,\ \mu_x(\chi)=1.
\end{equation*}
Therefore, $\mathbf{P^*} = \mathbf{P_{\mu_q}^*}$.
\end{proof}

\begin{figure}
  \centering
  \includegraphics[keepaspectratio=true,scale=0.75]{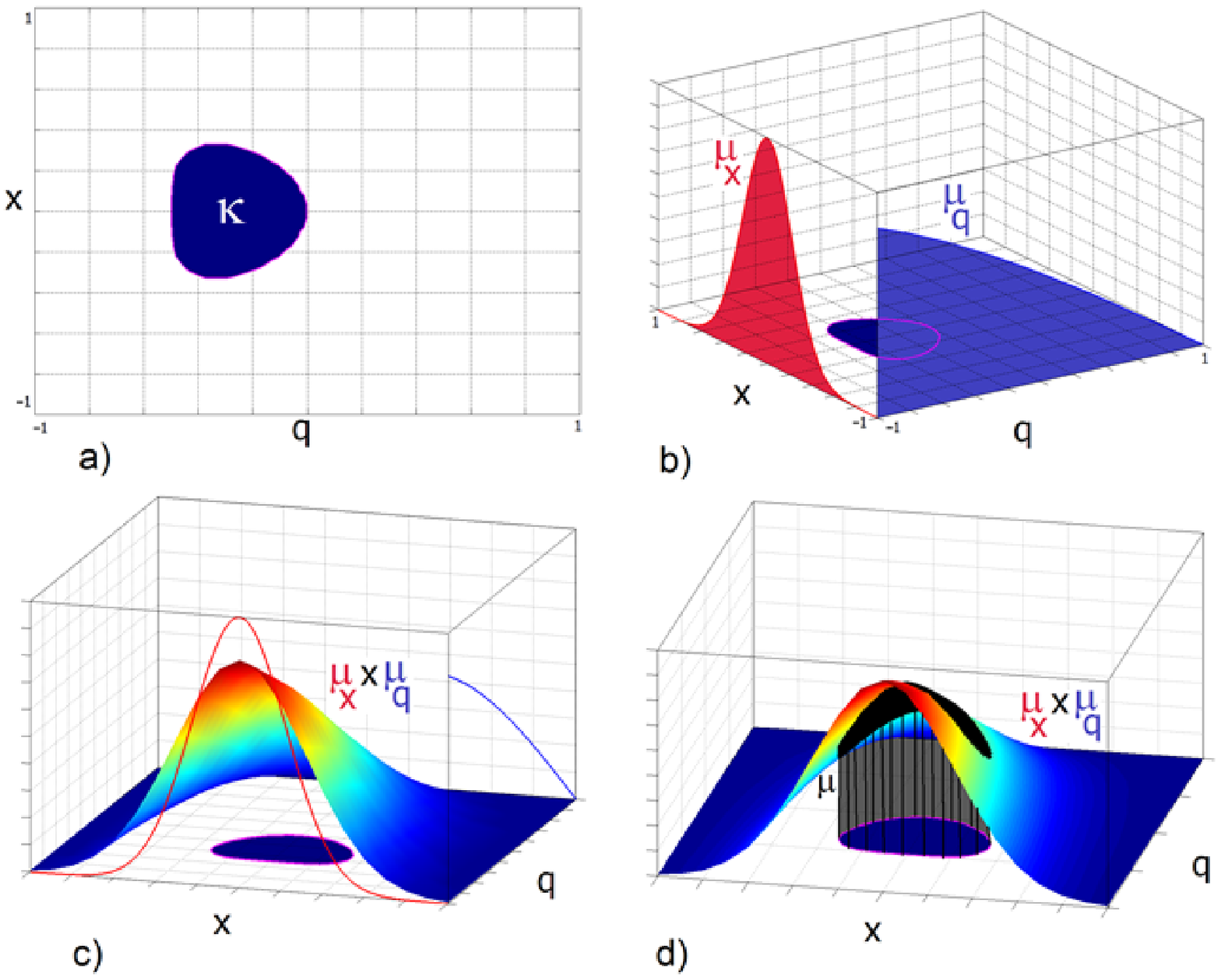}
\caption{\textbf{a)} Simple chance optimization problem over semialgebraic set $\mathcal K$ with random parameter $q$, and decision variable $x$, \textbf{b)} Equivalent problem in the measure space over probability measure $\mu_x$ as variable for given probability measure $\mu_q$, \textbf{c)} Probability of given semi algebraic set $\mathcal K$ for a fixed $\mu_x$ is equal to the integral of $\cK$ with respect to the measure $\mu_x \times \mu_q$, \textbf{d)} The probability is equal to the volume of the measure $\mu$ which is supported on the set $\mathcal K$ and has the same distribution as the measure $\mu_x \times \mu_q$ over its support}
\label{fig chance }
\vspace{-0.75cm}
\end{figure}

As an example, consider the following chance constrained problem corresponding to the semialgebraic set $\mathcal K$, displayed in Fig.1.a, in the space of $(x,q)\in\reals\times\reals$. Our objective is to compute an optimal decision $x^*$ that attains 
$\mathbf{P^*}=\sup_{x\in[-1,1]}\mu_q(\cF(x))$, in presence of random variable $q$ with known probability measure $\mu_q$ supported on $[-1,1]$. In other words, $x^*$ should be chosen such that the probability of the random point $(x^*,q)$ belonging to $\mathcal K$ becomes maximum. Fig.1.b shows the problem in the measure space, where a probability measure $\mu_x$ is assigned to decision variable $x$. If $x\in[-1,1]$ is chosen randomly according to fixed $\mu_x$, then to calculate the probability of the random event $(x,q)\in\cK$, one should compute an integral with respect to measure $\mu_x \times \mu_q$ over the set $\mathcal K$ as in~\eqref{max_p2_proof_p2_2} -- see (Fig.1.c). This integral is equal to the volume of a measure which is supported on $\mathcal K$ and has the same distribution as $\mu_x \times \mu_q$ on $\cK$ -- see (Fig.1.d). Hence, for fixed $\mu_x$, one needs to look for the measure $\mu$ supported on $\mathcal K$ with maximum volume, and bounded above with measure $\mu_x \times \mu_q$. Therefore, searching for $\mu_x$ and $\mu$ simultaneously leads to the optimization problem~\eqref{max_p2} in the measure space.
\subsection{ Semidefinite Relaxations}
In this section, we provide an infinite dimensional SDP of which feasible region is defined over real sequences in $\seq$. Unlike the problem \eqref{max_p2} in which we are looking for measures, in the SDP formulation given in \eqref{max_p3}, we aim at finding moment sequences corresponding to a measure that is optimal to \eqref{max_p2}. After proving the equivalence of \eqref{max_p2} and \eqref{max_p3}, we next provide a sequence of finite dimensional SDPs and show that the corresponding sequence of optimal solutions can arbitrarily approximate the optimal solution of \eqref{max_p3}, which characterizes the optimal solution of \eqref{max_p2}.

Consider the following infinite dimensional SDP: 
\begin{align}
\mathbf{P^*_{y_q}}:= &\sup _{\mathbf y,\mathbf y_x\in\seq} (\mathbf{y})_\mathbf{0}, \label{max_p3}\\
\hbox{s.t.}\quad & M_{\infty }(\mathbf y)\succcurlyeq 0,\ M_\infty (\mathbf y;\mathbf{p}_j)\succcurlyeq 0,\quad j=1,\dots ,\ell, \label{eq:y_cons_inf}\subeqn\\
&M_\infty ({\mathbf y}_{\mathbf x})\succcurlyeq 0,\ \Vert {\mathbf y}_{\mathbf x} \Vert_\infty \leq 1,\ \left(\bf{y_x}\right)_\mathbf{0}=1, \label{eq:x_cons_inf}\subeqn\\
&M_\infty (\mathbf{A}{\mathbf y}_{\mathbf x}-{\mathbf y})\succcurlyeq 0,\label{eq:moment_ineq}\subeqn
\end{align}
where $\mathbf{A}:\seq\rightarrow\seq$ is a linear map depending only on $\mu_q$. Indeed, let $\mathbf{y_q}:=\{y_{q_\beta}\}_{\beta\in\mathbb{N}^m}$ be the moment sequence of $\mu_q$. Then for any given $\mathbf{y_x}=\{y_{x_\alpha}\}_{\alpha\in\mathbb{N}^n}$, $\mathbf{Ay_x}=\mathbf{\bar{y}}$ such that $(\mathbf{\bar{y}})_\theta=(\mathbf{y_q})_\beta (\mathbf{y_x})_\alpha$ for all $\theta=(\beta,\alpha)\in\mathbb{N}^m\times\mathbb{N}^n$. Given $\mathbf{y}\in\seq$, $M_\infty(\mathbf{y})\succcurlyeq 0$ means that $M_d(\mathbf{y})\succcurlyeq 0$ for all $d\in\integers_+$.

The following lemma establishes the equivalence of \eqref{max_p2} and \eqref{max_p3}.
\begin{lemma}
\label{max lem 1}
Suppose that $\cK$ satisfies Assumption~\ref{assump:putinar}. If an optimal solution to \eqref{max_p2} exists, call it $(\mu^*,~\mu^*_x)$, then their moment sequences $(\mathbf y^*,\mathbf y_x^*)$ is an optimal solution to \eqref{max_p3}. Conversely, if an optimal solution to~\eqref{max_p3} exists, call it $(\mathbf y^*,\mathbf y_x^*)$, then there exists \emph{representing} measures $\mu^*$ and $\mu^*_x$ such that $(\mu^*,~\mu^*_x)$ is optimal to \eqref{max_p2}. Moreover, the optimal values of \eqref{max_p2} and \eqref{max_p3} are the same, i.e., $\mathbf{P_{\mu_q}^*}=\mathbf{P_{y_q}^*}$.
\end{lemma}
\begin{proof}
Suppose that $(\mu,\mu_x)$ is feasible to \eqref{max_p2}. Let $\mathbf{y}$ and $\mathbf{y_x}$ be the moment sequences corresponding to $\mu$ and $\mu_x$, respectively. Lemma~\ref{preliminary result_measure 2} implies \eqref{eq:y_cons_inf}; Lemma~\ref{lemma1} and Lemma~\ref{lem:nec_cond} imply \eqref{eq:x_cons_inf}. Moreover, let $\bar{\mathbf y}=\{ \bar{y}_\alpha\}_{\alpha\in\mathbb{N}^{n+m}}$ be the moment sequence corresponding to the product measure $\bar{\mu}:=\mu_x \times \mu_q$. \eqref{eq:mesuare_ineq} implies that $\bar{\mu}-\mu$ is a measure; hence, Lemma~\ref{lem:nec_cond} implies $M_\infty (\bar{\mathbf y}-\mathbf y)\succcurlyeq 0$. 
Moreover, the definition of $\mathbf{A}$ implies that $\bar{\mathbf{y}}=\mathbf{A y_x}$, which gives \eqref{eq:moment_ineq}. Since $\mathbf{y}$ is chosen to be the moment sequence of $\mu$, we have $\int d\mu=y_\mathbf{0}$. This shows that for each $(\mu,\mu_x)$ feasible to \eqref{max_p2}, one can construct a feasible solution to \eqref{max_p3} with the same objective value. Therefore, $\mathbf{P^*_{y_q}}\geq \mathbf{P^*_{\mu_q}}$. Note that Assumption~\ref{assump:putinar} is not used for this argument.

Next, suppose that $\mathbf{(y,y_x)}$ is a feasible solution to \eqref{max_p3}. Since $\cK$ satisfies Assumption~\ref{assump:putinar}, \eqref{eq:y_cons_inf} and Lemma~\ref{preliminary result_measure 2} together imply that $\mathbf{y}$ has a representing finite Borel measure $\mu$ supported on $\cK$, i.e., $\mu\in\cM(\cK)$. Moreover, \eqref{eq:x_cons_inf} and Lemma~\ref{lem:suf_cond_prob} together imply that $\mathbf{y_x}$ has a representing probability measure $\mu_x$ supported on hyper-cube $\chi$, i.e., $\mu_x\in\cM(\chi)$ such that $\mu_x(\chi)=1$.
Hence, the sequence $A\mathbf{y_x}$ has a representing measure $\bar{\mu}$ which is the product measure of $\mu_x$ and $\mu_q$, i.e., $\bar{\mu}=\mu_x\times\mu_q$. Furthermore, since $\cK\subset\chi\times\cQ=[-1,1]^{n+m}$, \eqref{eq:moment_ineq} implies that $\mu\preceq\bar{\mu}$, which is \eqref{eq:mesuare_ineq}. Finally, the fact that $\mu$ is a representing measure of $\mathbf{y}$ implies that $\int d\mu=y_\mathbf{0}$. Therefore, $\mathbf{P^*_{y_q}}\leq\mathbf{P^*_{\mu_q}}$. Combining this with the above result gives us $\mathbf{P^*_{y_q}}=\mathbf{P^*_{\mu_q}}$.
\end{proof}

In order to have tractable approximations to the infinite dimensional SDP in \eqref{max_p3},
we consider the following sequence of SDPs, known as Lasserre's hierarchy~\cite{r_measure_ref 2}, defined below: 
\begin{align}
\mathbf{P}_d:= &\sup_{\mathbf{y}\in\reals^{S_{n+m,2d}},\ \mathbf{y_x}\in\reals^{S_{n,2d}}} (\mathbf{y})_\mathbf{0}, \label{max_p4}\\
\hbox{s.t.}\quad & M_d(\mathbf y)\succcurlyeq 0,\ M_{d-r_j}(\mathbf{y}; \mathbf{p}_j)\succcurlyeq 0,\quad j=1,\dots ,\ell, \label{eq:y_cons_i}\subeqn\\
&M_d ({\mathbf y}_{\mathbf x})\succcurlyeq 0,\ \norm{\mathbf{y_x}}_\infty \leq 1,\ \left(\mathbf{y_x}\right)_\mathbf{0}=1, \label{eq:x_cons_i}\subeqn\\
&M_d (A_d\mathbf{y_x}-{\mathbf y})\succcurlyeq 0,\label{eq:moment_ineq_i}\subeqn
\end{align}
where $\delta_j$ is the degree of $\cP_j$, $r_j:=\left\lceil\frac{\delta_j}{2}\right\rceil$ for all $1\leq j\leq \ell$, and $A_d:\reals^{S_{n,2d}}\rightarrow\reals^{S_{n+m,2d}}$ is defined similarly to $\mathbf{A}$ in \eqref{max_p3}. Indeed, let $\mathbf{y_q}:=\{y_{q_\beta}\}_{\beta\in\mathbb{N}_{2d}^m}$ be the truncated moment sequence of $\mu_q$. Then for any given $\mathbf{y_x}=\{y_{x_\alpha}\}_{\alpha\in\mathbb{N}_{2d}^n}$, $A_d\mathbf{y_x}=\mathbf{\mathbf{y}}$ such that $(\mathbf{\bar{y}})_\theta=(\mathbf{y_q})_\beta (\mathbf{y_x})_\alpha$ for all $\theta=(\beta,\alpha)\in\mathbb{N}_{2d}^{n+m}$.


In the following theorem, it is shown that the sequence of optimal solutions to the SDPs in \eqref{max_p4} converges to the solution of the infinite dimensional SDP in \eqref{max_p3}. In essence, the following theorem is similar to Theorem~3.2 in~\cite{r_vol}; however, for the sake of completeness we give its proof below.
\begin{theorem} \label{Max theo 2}
For all $d\geq 1$, there exists an optimal solution $(\mathbf{y}^d,\mathbf{y}^d_\mathbf{x})\in\reals^{S_{n+m,2d}}\times\reals^{S_{n,2d}}$ to \eqref{max_p4} with the optimal value $\mathbf{P}_d$. Let $\cS:=\{(\mathbf{y}^d,\mathbf{y}^d_\mathbf{x})\}_{d\in\integers_+}\subset\seq\times\seq$ be such that each element of $\cS$ is obtained by zero-padding, i.e., $\left(\mathbf{y}^d\right)_\alpha=0$ for all $\alpha\in\mathbb{N}^{n+m}$ such that $\norm{\alpha}_1>2d$, and $\left(\mathbf{y}^d_{\mathbf{x}}\right)_\alpha=0$ for all $\alpha\in\mathbb{N}^{n}$ such that $\norm{\alpha}_1>2d$. Then $\{\mathbf{P}_d\}_{d\in\integers_+}$ and $\cS$ have the following properties:
\begin{enumerate}[i)]
\item $\lim_{d\in\integers_+}\mathbf{P}_d=\mathbf{P^*}$, the optimal value of \eqref{intro_max},
\item
There exists an accumulation point of $\cS$ in the weak-$\star$ topology of $\ell_\infty$ and every accumulation point of $\cS$ is an optimal solution to \eqref{max_p3}. Hence, there exists corresponding representing measures $(\mu^*,\mu^*_x)$ that is optimal to \eqref{max_p2} and any $x^*\in supp(\mu_x^*)$ is optimal to \eqref{intro_max}.
\end{enumerate}
\end{theorem}
\begin{proof}
First, we will show that for all $d\geq 1$, the corresponding feasible region of \eqref{max_p4} is bounded . Fix $d\geq 1$.
Let $(\mathbf{y,y_x})$ be a feasible solution to \eqref{max_p4}. Then from \eqref{eq:x_cons_i}, we have $\norm{\mathbf{y_x}}_\infty\leq 1$. Since $\mu_q$ is a probability measure supported on $\cQ=[-1,1]^m$, Lemma~\ref{lemma1} implies that $\norm{\mathbf{y_q}}_\infty\leq 1$ as well. Moreover, the definition of $A_d$ further implies that $\norm{A_d\mathbf{y_x}}_\infty\leq 1$. Let $\mathbf{\bar{y}}:=A_d\mathbf{y_x}$.  
It follows from \eqref{eq:moment_ineq_i} that the diagonal elements of $M_d(\mathbf{\bar{y}}-\mathbf{y})$ are nonnegative, i.e., $\left(\mathbf{\bar{y}}\right)_{2\alpha}-\left(\mathbf{y}\right)_{2\alpha}\geq 0$ for all $\alpha\in\mathbb{N}_d^{n+m}$. This implies that
\begin{equation}
\label{eq:moment_bound}
\max \left\{ y_\mathbf{0}, \max_{i=1,\ldots,n+m} L_{\mathbf{y}}\left(x_i^{2d}\right) \right\}\leq\max_{\alpha\in\mathbb{N}_d^{n+m}}y_{2\alpha}\leq\max_{\alpha\in\mathbb{N}_d^{n+m}}\bar{y}_{2\alpha}\leq\norm{\mathbf{\bar{y}}}_\infty\leq 1,
\end{equation}
where the first inequality follows from the fact that $$\{y_\mathbf{0}\}\cup \left\{L_{\mathbf{y}}\left(x_i^{2d}\right):\ i=1,\ldots,n+m \right\}\subset\{y_{2\alpha}:\ \alpha\in\mathbb{N}_d^{n+m}\}.$$
From \eqref{eq:y_cons_i}, we have $M_d(\mathbf y)\succcurlyeq 0$. Hence, using Lemma~\ref{mom_bound 1}, \eqref{eq:moment_bound} implies that $|y_\alpha|\leq \norm{\mathbf{\bar{y}}}_\infty\leq 1$ for all $\alpha\in\mathbb{N}^{n+m}_{2d}$. Therefore, the feasible region is bounded. Since the cone of positive semidefinite matrices is a closed set and all the mappings in \eqref{max_p4} is linear, we also conclude that the feasible region is compact. Hence, there exists an optimal solution $(\mathbf{y}^d,\mathbf{y}^d_{\mathbf{x}})$ to the problem \eqref{max_p4} for all $d\geq 1$.

Fix $d\geq 1$. Clearly, for any given feasible solution $(\mathbf{y,y_x})$ to \eqref{max_p3}, by truncating the both sequences to vectors $\mathbf{y}\in\reals^{S_{n+m,2d}}$ and $\mathbf{y_x}\in\reals^{S_{n,2d}}$, we can construct a feasible solution to \eqref{max_p4} with the same objective value. Hence, it can be concluded that $\mathbf{P}_d\geq \mathbf{P^*_{y_q}}$ for all $d\geq 1$. Moreover, the same argument also shows that $\mathbf{P}_d\geq\mathbf{P}_{d'}$ for all $d'\geq d$. Hence, $\{\mathbf{P}_d\}_{d\in\integers_+}$ is a decreasing sequence bounded below by $\mathbf{P^*_{y_q}}$. Therefore, it is convergent and has a limit such that $\lim_{k\in\integers_+}\mathbf{P}_k\geq\mathbf{P^*_{y_q}}$.

In order to collect all the optimal solutions corresponding to different $d$ in one space, we
extend $(\mathbf{y}^d,\mathbf{y}^d_{\mathbf{x}})\in\reals^{S_{n+m,2d}}\times\reals^{S_{n,2d}}$ to vectors in $\ell_\infty$ (the Banach space of bounded sequences equipped with the sup-norm) by zero-padding, i.e., we set $(\mathbf{y}^d)_\alpha=0$ for all $\alpha\in\mathbb{N}^{n+m}$ such that $\norm{\alpha}_1>2d$, and $\left(\mathbf{y}^d_{\mathbf{x}}\right)_\alpha=0$ for all $\alpha\in\mathbb{N}^{n}$ such that $\norm{\alpha}_1>2d$. Note that $\ell_\infty$ is the dual space of $\ell_1$, which is separable; hence, sequential Banach-Alaoglu theorem states that the closed unit ball of $\ell_\infty$, denoted by $\cB_\infty$, is weak-$\star$ sequentially compact. 
Since $\{\mathbf{y}^d\}_{d\in\integers_+}\subset\cB_\infty$ and $\{\mathbf{y}^d_{\mathbf{x}}\}_{d\in\integers_+}\subset\cB_\infty$, there exists a subsequence $\{d_k\}\subset\integers_+$ such that $\{\mathbf{y}^{d_k}\}_{k\in\integers_+}$ and $\{\mathbf{y}^{d_k}_{\mathbf{x}}\}_{k\in\integers_+}$ converge weak-$\star$ to $\mathbf{y}^*\in\cB_\infty$ and $\mathbf{y^*_x}\in\cB_\infty$ in the weak-$\star$ topology, respectively. Hence,
\begin{equation}
\lim_{k\in\integers_+}\left(\mathbf{y}^{d_k}\right)_\alpha=\left(\mathbf{y}^*\right)_\alpha,\quad \forall~\alpha\in\mathbb{N}^{n+m},\quad\quad\lim_{k\in\integers_+}\left(\mathbf{y}^{d_k}_{\mathbf{x}}\right)_\alpha=\left(\mathbf{y^*_x}\right)_\alpha,\quad \forall~\alpha\in\mathbb{N}^{n}.
\end{equation}
Fix $d\geq 1$, then for all $k\in\integers_+$ such that $d_k\geq d$, we have
\begin{align*}
& M_d(\mathbf{y}^{d_k})\succcurlyeq 0,\ M_{d-r_j}(\mathbf{y}^{d_k};\mathbf{p}_j)\succcurlyeq 0,\quad j=1,\dots ,\ell,\\
&M_d (\mathbf{y}^{d_k}_\mathbf{x})\succcurlyeq 0,\ \norm{\mathbf{y}^{d_k}_\mathbf{x}}_\infty \leq 1,\ \left(\mathbf{y}^{d_k}_\mathbf{x}\right)_\mathbf{0}=1,\\
&M_d (\mathbf{A}\mathbf{y}^{d_k}_\mathbf{x}-\mathbf{y}^{d_k})\succcurlyeq 0.
\end{align*}
Since $d\in\integers_+$ is arbitrary, by taking the limit as $k\to\infty$, we see that $(\mathbf{y}^*,\mathbf{y}^*_\mathbf{x})$ satisfies all the constraints in \eqref{max_p3}. Therefore, $(\mathbf{y^*})_\mathbf{0}\leq \mathbf{P^*_{y_q}}$. On the other hand, $(\mathbf{y^*})_\mathbf{0}=\lim_{k\in\integers_+}\left(\mathbf{y}^{d_k}\right)_\mathbf{0}=\lim_{k\in\integers_+}\mathbf{P}_{d_k}$. Moreover, since every subsequence of a convergent sequence converges to the same point, we have $\lim_{k\in\integers_+}\mathbf{P}_k=\lim_{k\in\integers_+}\mathbf{P}_{d_k}=\mathbf{P^*_{y_q}}$. This shows that the subsequential limit $(\mathbf{y}^*,\mathbf{y}^*_\mathbf{x})$ is an optimal solution to \eqref{max_p3}. The rest of the claims follow from our previous results:
Theorem~\ref{Max Theo 1} and Lemma~\ref{max lem 1}.

\end{proof}
\subsection{Discussion on Improving Estimates of Probability}
\label{sec:prob_approx}
In our numerical experiments, we have observed that the convergence of the upper bound $\mathbf{P}_d$ 
to the optimum probability $\mathbf{P}^*$ was slow in $d$ when we solved the sequence of SDP relaxations in \eqref{max_p4}. Suppose that the semi-algebraic set $\cK:=\{(x,q):\ \cP_j(x,q)\geq 0,\ j=1,\ldots,\ell\}$ satisfies Putinar's property. The procedure detailed below helped us to get better estimates on the optimum probability $\mathbf{P}^*$. To make the upcoming discussion easier we make the following assumptions: i) there is a \emph{unique} $x^*\in\Pi_1$ such that $\mu_q(\cF(x^*))=\mathbf{P}^*$, and there exists some $\bar{q}$ such that $(x^*,\bar{q})\in\relint\cK$, where $\cF$ is defined in \eqref{eq:F_func}, and $\Pi_1:=\{x\in\reals^n:\ \exists q\in\reals^m \hbox{ s.t. } (x,q)\in\cK\}\subset\chi:=[-1,1]^n$; and ii) $\mu_q\in\cM(\cQ)$ has the following ``\emph{continuity}" property: if $\{S_k\}\subset\Sigma_q$ such that $\lim_{k\rightarrow\infty}S_k=S^*$ in the Hausdorff-metric, then $\lim_{k\rightarrow\infty}\mu_q(S_k)=\mu_q(S^*)$. Let $({\bf y}^d,{\bf y}^d_{\bf{x}})$ denote an optimal solution to the SDP relaxation in~\eqref{max_p4}, and form $x^d\in\reals^n$ using the components of $({\bf y}^d_{\bf{x}})_\alpha$ such that $\norm{\alpha}_1=1$. Clearly, $x^d\in\chi$. Since $\mu_q\in\cM(\cQ)$ is given, we approximate the volume
$\int_{\cF(x^d)}d \mu_q$ as described in \cite{r_vol} by solving an SDP relaxation for
\begin{equation}
\label{eq:volume}
\mathbf{\bar{P}}_d:=\sup_{\mu'\in\cM(\cF(x^d))} \int d\mu' \hbox{ s.t. } \mu'\preceq \mu_q.
\end{equation}
Note that this intermediate SDP can be built only after the relaxation in \eqref{max_p4} is solved. Let $\mathbf{P}'_d$ denote the optimal value of the volume approximation SDP corresponding to \eqref{eq:volume} with relaxation order $d$. Clearly, $\mathbf{\bar{P}}_d=\mu_q\left(\cF(x^d)\right)\geq 0$, and for all $d$ we have $\mathbf{P}_d\geq\mathbf{P}^*\geq\mathbf{\bar{P}}_d$,  and $\mathbf{P}_d\geq\mathbf{P}'_d\geq\mathbf{\bar{P}}_d$.
Note that since $x^*$ is the unique optimal solution (\emph{assumption i}), Theorem~\ref{Max theo 2} implies that $\lim_{d\rightarrow\infty}({\bf y}^d_{\bf x})_\alpha=({\bf y^*_x})_\alpha$ for all $\alpha\in\mathbb{N}^{n}$ such that ${\bf y^*_x}$ is the moment sequence corresponding to Dirac measure at $x^*$. Therefore, from the definition of $x^d$, it follows that $\lim_{d\rightarrow\infty}x^d=x^*$.
Also note that since $\cK$ is compact (from Putinar's property) and $\cP_j$ is a polynomial in $(x,q)$ for all $j=1,\ldots,\ell$, it follows that the multifunction $\cF:\chi\rightarrow\Sigma_q$ such that $\cF(x)=\{q\in\cQ:\ (x,q)\in\cK\}$ with $\dom\cF=\Pi_1$ is locally bounded, closed-valued, and $\lim_{d\rightarrow\infty}\cF(x^d)=\cF(x^*)$ in Hausdorff metric. Hence, \emph{assumption ii} implies that $\lim_{d\rightarrow\infty}\mathbf{\bar{P}}_d=\lim_{d\rightarrow\infty}\mu_q\left(\cF(x^d)\right)=\mathbf{P}^*$. Moreover, since $\lim_{d\rightarrow\infty}\mathbf{P}_d=\mathbf{P}^*$ (from Theorem~\ref{Max theo 2}), and $\mathbf{P}_d\geq\mathbf{P}'_d\geq\mathbf{\bar{P}}_d$ for all $d$, we can conclude that $\lim_{d\rightarrow\infty}\mathbf{P}'_d=\mathbf{P}^*$ as well.

We noticed in our numerical experiments that although $\{\mathbf{P}'_d\}_{d\in\integers_+}$ is closer to $\mathbf{P}^*$ when compared to $\{\mathbf{P}_d\}_{d\in\integers_+}$, the convergence of $\mathbf{P}'_d$ to $\mathbf{P}^*$ was still slow in practice as $d$ increases.
This phenomena may partly be explained as in \cite{r_vol} by considering the dual problem. Let $\cC$ be the Banach space of continuous functions on $\cQ$ such that $\norm{f}:=\sup_{q\in\cQ}f(q)$ for $f\in\cC$, and $\cC_+:=\{f\in\cC:\ f\geq 0 \hbox{ on } \cQ\}$. The Lagrangian dual of \eqref{eq:volume} is given below:\vspace{-0.5cm}
\begin{align}
\mathbf{\bar{P}}_d^\mathbf{Dual}:=&\ \inf_{f \in \cC_+} \int f~d\mu_q, \label{max_Dual}\\
\hbox{s.t.}\quad & f\geq 1 \hbox{ on } \cF(x^d). \nonumber
\end{align}
Moreover, \emph{assumption ii} (``continuity" of $\mu_q$) and Urysohn's Lemma together imply that $\mathbf{\bar{P}}_d^\mathbf{Dual}=\mathbf{\bar{P}}_d$ for all $d$. 
Let $\cI_{\cF(x^d)}$ denote the indicator function of the semi-algebraic set $\cF(x^d)$, i.e., $\cI_{\cF(x^d)}(q)=1$ if $q\in\cF(x^d)$, and $0$ otherwise. Indeed, solving the SDP relaxation of \eqref{eq:volume} corresponds in dual space to approximating $\cI_{\cF(x^d)}$, which is \emph{discontinuous} on the boundary of the set. Therefore, although there exists a minimizing sequence of functions belonging to $\cC_+$ that approximates $\cI_{\cF(x^d)}$ from above, the discontinuity on the boundary of $\cF(x^d)$ causes the Gibbs phenomenon -- see the oscillation observed in Figure~3.3.a. This might be an important factor lurking behind the numerically observed slow convergence of $\{\mathbf{P}'_d\}_{d\in\integers_+}$ to $\mathbf{P}^*$.

Let $\cG^d:\cQ\rightarrow\reals$ such that $\cG^d(q):=\prod_{j=1}^\ell \cP_j(x^d,q)$. To deal with the numerical problems caused by approximating the discontinuous indicator function, we propose to solve
\begin{equation}
\label{eq:volume_new}
\sup_{\tilde{\mu}\in\cM(\cF(x^d))} \int \cG^d~d\tilde{\mu} \hbox{ s.t. } \tilde{\mu}\preceq \mu_q.
\end{equation}
Let $\mu^*_d$ denote the optimal solution to \eqref{eq:volume_new}. ``Continuity" of $\mu_q$ in \textit{assumption ii} implies that $\cG^d$ is strictly positive almost everywhere on $\cF(x^d)$. Hence, $\mu^*_d$ is clearly also optimal to \eqref{eq:volume}. Therefore, $\mu^*_d\left(\cF(x^d)\right)=\mu_q(\cF(x^d))=\mathbf{\bar{P}}_d\rightarrow\mathbf{P}^*$ as $d\rightarrow\infty$. Let $\cU^d(q)=\max\{\cG^d(q),~0\}$, and note that $\cU^d$ is continuous on the boundary of $\cF(x^d)$; hence, it is important to emphasize that solving \eqref{eq:volume_new} corresponds to approximating the \emph{continuous} function $\cU^d$ from above on $\cF(x^d)$. These properties of \eqref{eq:volume_new} motivated us to numerically investigate the behaviour of $\{\mathbf{\tilde{P}}_d\}_{d\in\integers_+}$ sequence, where $\mathbf{\tilde{P}}_d:= ({\bf \tilde{y}}^d)_0$ and ${\bf \tilde{y}}^d$ denotes an optimal solution to the SDP relaxation for \eqref{eq:volume_new} with order $d$. In our numerical experiments we observed that $\mathbf{\tilde{P}}_d\rightarrow\mathbf{P}^*$; however, this time with a faster convergence rate. 
To illustrate this behavior numerically, we considered two simple example problems in Section~\ref{sec:simple_example}. 
\subsection{Simple Examples}
\label{sec:simple_example}
In this section, we present two simple example problems that illustrate the effectiveness of
the proposed methodology 
to solve the chance optimization problem in \eqref{intro_max}. 
The decision variables and the uncertain problem parameters in these examples are low dimensional for illustrative purposes. In the first example, we considered a problem over a semialgebraic set defined by a single polynomial:\vspace{-0.15cm}
\begin{equation} \label{max_exa1_def1}
\sup_{x\in\reals} \mu_q\left(\{q \in\reals :\ \mathcal{P}(x,q)\geq 0\ \}\right),
\end{equation}
\vspace{-0.5cm}
where
\begin{equation} \label{max_exa1_def2}
\mathcal{P}\left(x,q\right)
=\tfrac{1}{2}q\left(q^2+{(x-\tfrac{1}{2})}^2\right)-\left(q^4+q^2{(x-\tfrac{1}{2})}^2+{(x-\tfrac{1}{2})}^4\right).
\end{equation}

The uncertain parameter $q\in\reals$ has a uniform distribution on [-1,1]. To obtain an approximate solution, we solve the SDP in \eqref{max_p4} with the minimum relaxation order $d=2$ since
the degree of the polynomial in \eqref{max_exa1_def2} is 4. The moment vectors $\mathbf y_{\mathbf q}$, $\mathbf{y}_\mathbf{x}$, and $\mathbf{y}$ for the measures $\mu_q$ and $\mu_x$, and $\mu$ up to order four are
\[\mathbf y_{\mathbf q}^T =\left[1,~0,~\tfrac{1}{3},~0,~\tfrac{1}{5}\right],\quad \mathbf{y}^T_\mathbf{x}=\left[1,~y_{x_1},~y_{x_2},~y_{x_3},~y_{x_4}\right],\]
\vspace{-8 mm}
\[ \mathbf{y}^T=\left[y_{00}~|~y_{10},~y_{01}~|~y_{20},~y_{11},~y_{02}~|~y_{30},~y_{21},~y_{12},~y_{03}~|~
y_{40},~y_{31},~y_{22},~y_{13},~y_{04}\right].
\]
Given moment vectors $\mathbf y_{\mathbf q} $, 
the moment vector $\mathbf{\bar{y}}$ for the measure $\overline{\mu} =\mu_x \times \mu_q $ has the form
\begin{small}
\begin{align*}
\mathbf{\bar{y}}^T &=
\left[1~|~y_{x_1},~y_{q_1}~|~y_{x_2},~y_{x_1}y_{q_1},~y_{q_2}~|~y_{x_3},~y_{x_2}y_{q_1},~y_{x_1}y_{q_2},~y_{q_3}~|~
y_{x_4},~y_{x_3}y_{q_1},~y_{x_2}y_{q_2},~y_{x_1}y_{q_3},~y_{q_4}\right],\\
&= \left[1~|~y_{x_1},~0~|~y_{x_2},~0,~\tfrac{1}{3}~|~y_{x_3},~0,~\tfrac{1}{3}y_{x_1},~0~|~y_{x_4},~0,~\tfrac{1}{3} y_{x_2},~0,~\tfrac{1}{5}\right].
\end{align*}
\end{small}
\noindent SDP in~\eqref{max_p4} with $d=2$ is solved using SeDuMi~\cite{sedumi}, which is an interior-point solver add-on for Matlab, and 
the following solution was obtained:
\begin{small}
\[{\mathbf y^*}^T =\left[0.66,~0.3,~0.14,~0.16,~0.07,~0.1,~0.08,~0.03,~0.05,~0.04,~0.04,~0.02,~0.02,~0.02,~0.02\right],\]
\vspace{-0.2in}
\[{\mathbf{y_x^*}}^T =[ 1,0.50,0.25,0.13,0.85].\]
\end{small}
\begin{figure}[!h]
\centering
 \includegraphics[scale=0.5]{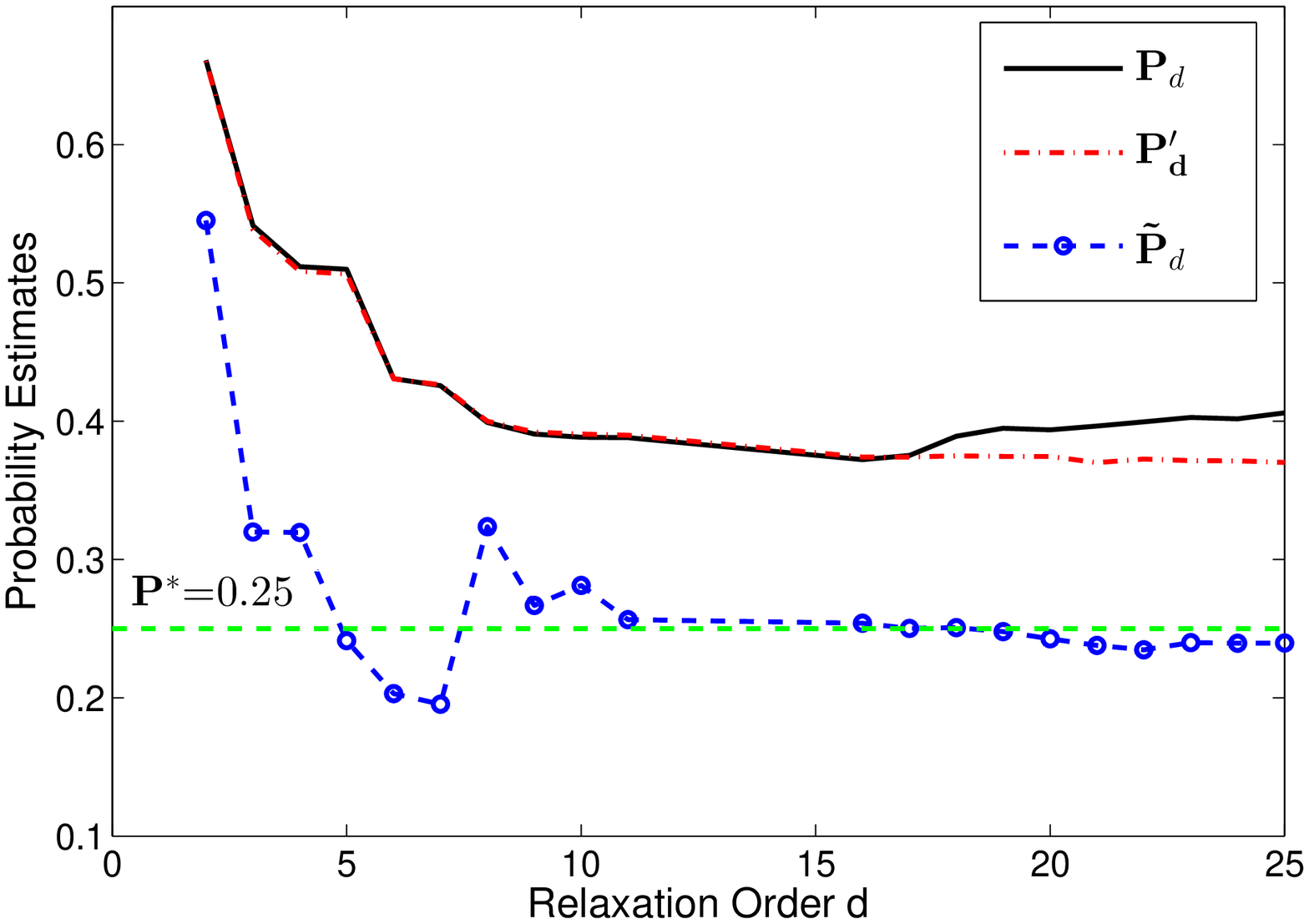}
\caption{$\mathbf{P}_d$, $\mathbf{P'}_d$, and $\mathbf{\tilde{P}}_d$ for increasing relaxation order $d$}
\label{fig:2}       
\vspace{-0.5cm}
\end{figure}

We approximate the solution to \eqref{intro_max} with $y^*_{x_1} = 0.5$ (in Section~\ref{sec:prob_approx} we make a case for this approximation under some simplifying assumptions), and estimate the optimal probability $\mathbf{P^*}$ with $\mathbf{P}_2=y^*_{00}= 0.66$. To test the accuracy of the results obtained, we used Monte Carlo simulation to estimate $\mathbf{P^*}$ and an optimal solution to \eqref{max_exa1_def1}. The details of the Monte Carlo simulation are discussed in Section~\ref{sec:simulation}. This computationally intensive method estimated that $x^* =0.5$ with optimal probability of $0.25$. 
To obtain better estimates of the optimum probability, one needs to increase the relaxation order $d$. 

Figure~\ref{fig:2} displays the three sequences defined in Section~\ref{sec:prob_approx}: $\{\mathbf{P}_d\}_{d\in\integers_+}$, $\{\mathbf{P}'_d\}_{d\in\integers_+}$, and $\{\mathbf{\tilde{P}}_d\}_{d\in\integers_+}$, against the optimal probability $\mathbf{P^*}=0.25$ denoted by the green dashed line. 
For increasing relaxation orders $d = 2,...,25$, we adopted SeDuMi~\cite{sedumi} to compute $\mathbf{P}_d$ and $\mathbf{P}'_d$, the optimal values of the SDP in \eqref{max_p4}, and of the SDP relaxation for the volume problem in \eqref{eq:volume} with relaxation order $d$, respectively; and also to compute $\mathbf{\tilde{P}}_d=({\bf \tilde{y}}^d)_0$. 
Similar to the results in~\cite{r_vol}, Figure~\ref{fig:2} shows a faster convergence to $\bf P^*$ for the case when $\int \cG^d~d\tilde{\mu}$ is maximized as in \eqref{eq:volume_new}. Let $\cI_{\cF(x^d)}$ denote the indicator function of $\cF(x^d)$, i.e., $\cI_{\cF(x^d)}(q)=1$ if $q\in\cF(x^d)$, and $0$ otherwise. As discussed in Section~\ref{sec:prob_approx}, $\cU^d=\max\{\cG^d, 0\}$ is a continuous function while $\cI_{\cF(x^d)}$ is discontinuous on the boundary of $\cF(x^d)$; and this might be a factor affecting the convergence speed. Indeed, Figure~\ref{fig:2-2}.a displays the degree-100 polynomial approximation $f^*$ to $\cI_{\cF(x^*)}$, the indicator function of the set $\cF(x^*)$, i.e., $f^*$ is a minimizer to $\inf_{f\in\mathbb R_{\rm d}[x]}\{\int f~d\mu_q:\ f\geq 0\ \hbox{ on } \cQ,\ f\geq 1\ \hbox{ on } \cF(x^*)\}$ for $d=100$. Note that this problem is a restriction of the Lagrangian dual problem for $\sup\{\int~d\mu':\ \mu\preceq\mu_q,\ \mu'\in\cM(\cF(x^*))\}$ --indeed, dual variable $f\in\cC$ is restricted to be in $\mathbb R_{\rm d}[x]$. 
On the other hand, Figure~\ref{fig:2-2}.b displays the degree-100 polynomial approximation $h^*$ to the piecewise-polynomial function $\cU=\max\{\cG, 0\}$, where $\cG(q)=\cP(x^*,q)$ and $h^*$ is a minimizer to $\inf_{h\in\mathbb R_{\rm d}[x]}\{\int h~d\mu_q:\ h\geq 0\ \hbox{ on } \cQ,\ h\geq \cG\ \hbox{ on } \cF(x^*)\}$ for $d=100$. Similarly, this problem is a restriction of the Lagrangian dual problem for $\sup\{\int \cG~d\tilde{\mu}:\ \tilde{\mu}\preceq\mu_q,\ \tilde{\mu}\in\cM(\cF(x^*))\}$. Note that Figure~\ref{fig:2-2} shows that it is easier to approximate the \emph{continuous} function $\cU=\max\{\cG, 0\}$ than the \emph{discontinuous} indicator function $\cI_{\cF(x^*)}$.
\begin{figure}[!h]
\subfigure[$f^*$: the degree-100 polynomial approximation to $\cI_{\cF(x^*)}$, indicator function of $\cF(x^*)$]{\includegraphics[scale=0.4]{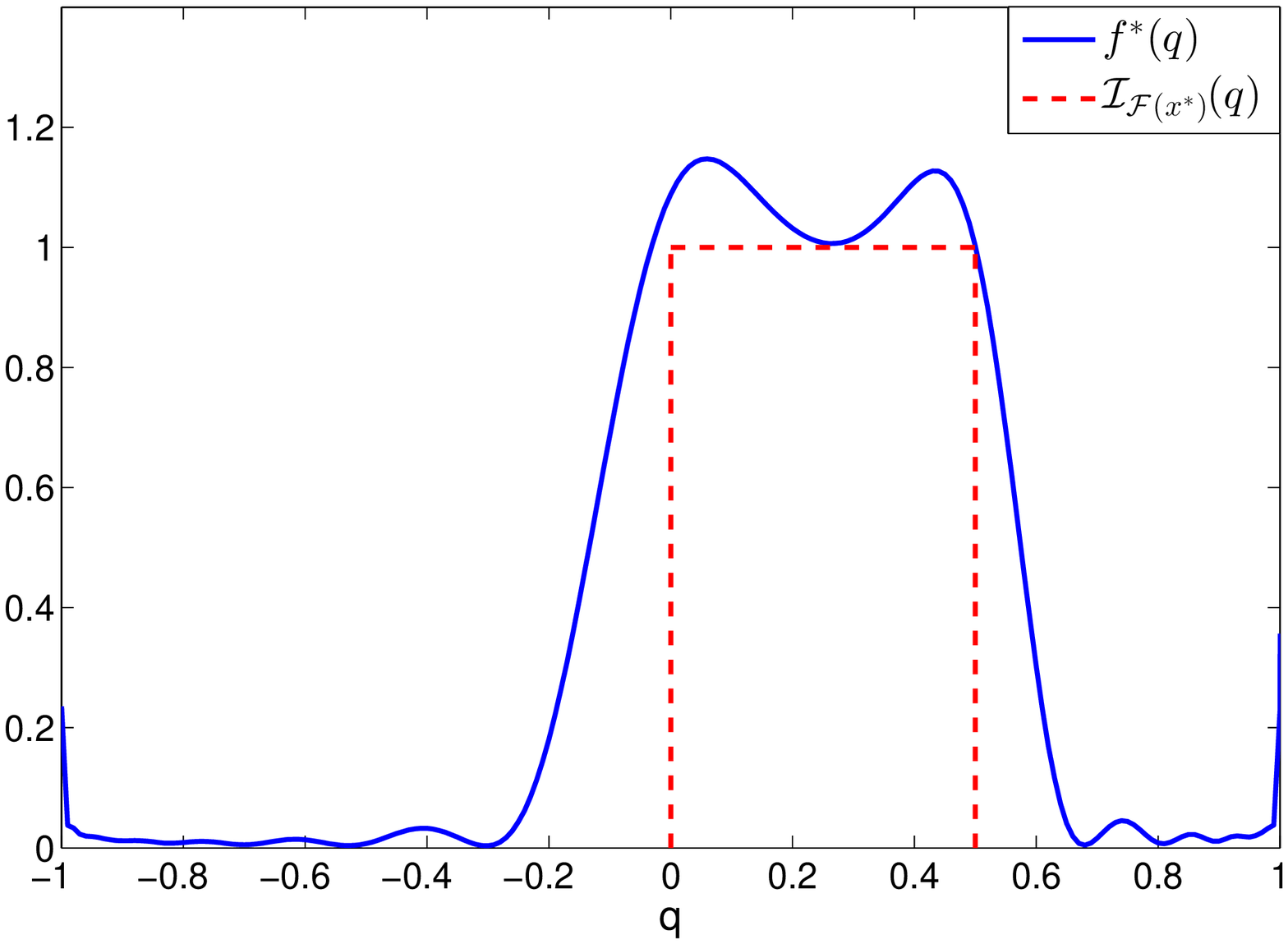}}
\subfigure[$h^*$: the degree-100 polynomial approximation of the piecewise-polynomial function $\cU(q)=\max\{\cG(q),~0\}$]{\includegraphics[scale=0.4]{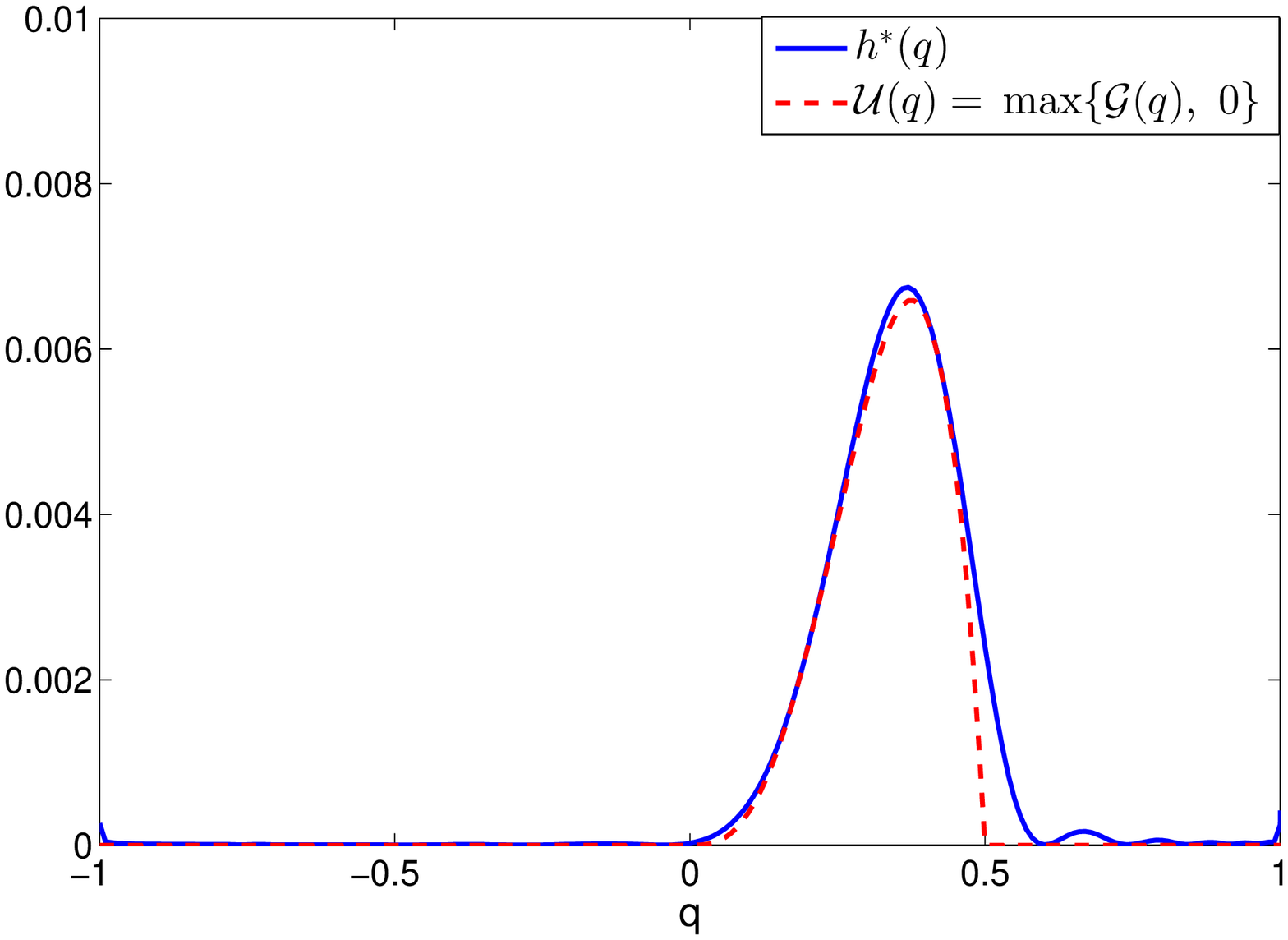}}
\caption{Comparison of $\sup_{\mu'\in\cM(\cF(x^*))} \int d\mu' \hbox{ s.t. } \mu'\preceq \mu_q$ and $\sup_{\tilde{\mu}\in\cM(\cF(x^*))} \int \cG~d\tilde{\mu} \hbox{ s.t. } \tilde{\mu}\preceq \mu_q$ from the dual perspective for $\cG(q)=\cP(x^*,q)$}
\label{fig:2-2}
\end{figure}

Next, we considered a problem over a semialgebraic set defined by an intersection of two polynomials:
\vspace{-0.2cm}
\begin{equation} \label{max_exa1_def3}
\sup_{x\in\reals} \mu_q\left(\{q \in\reals :\ \mathcal{P}_1(x,q)\geq0,\ \mathcal{P}_2(x,q)\geq0\ \}\right),
\vspace{-0.3cm}
\end{equation}
\vspace{-0.5cm}
where
\begin{equation} \label{max_exa1_def4}
\mathcal{P}_1\left(x,q\right)= 0.1275+0.7x-x^2-q^2,\quad  \mathcal{P}_2\left(x,q\right)= -0.1225+0.7x+q-x^2-q^2.
\end{equation}
The uncertain parameter $q\in\reals$ has a uniform distribution on [-1,1]. Against the optimal probability $\mathbf{P^*}=0.25$ denoted by the green dashed line, Figure~\ref{fig:22} displays two other sequences, $\{\mathbf{\tilde{P}^{(1)}}_d\}_{d\in\integers_+}$ and $\{\mathbf{\tilde{P}^{(2)}}_d\}_{d\in\integers_+}$, in addition to the three sequences defined in Section~\ref{sec:prob_approx}: $\{\mathbf{P}_d\}_{d\in\integers_+}$, $\{\mathbf{P}'_d\}_{d\in\integers_+}$, and $\{\mathbf{\tilde{P}}_d\}_{d\in\integers_+}$. Here, $\mathbf{\tilde{P}^{(1)}}_d$ and $\mathbf{\tilde{P}^{(2)}}_d$ are defined similarly to $\mathbf{\tilde{P}}_d=({\bf \tilde{y}^d})_\mathbf{0}$ by replacing $\cG^d(q)=\cP_1(x^d,q) \cP_2(x^d,q)$ in \eqref{eq:volume_new} with $\cP_1(x^d,q)$, and $\cP_2(x^d,q)$, respectively.
\begin{figure}
\centering
\includegraphics[scale=0.5]{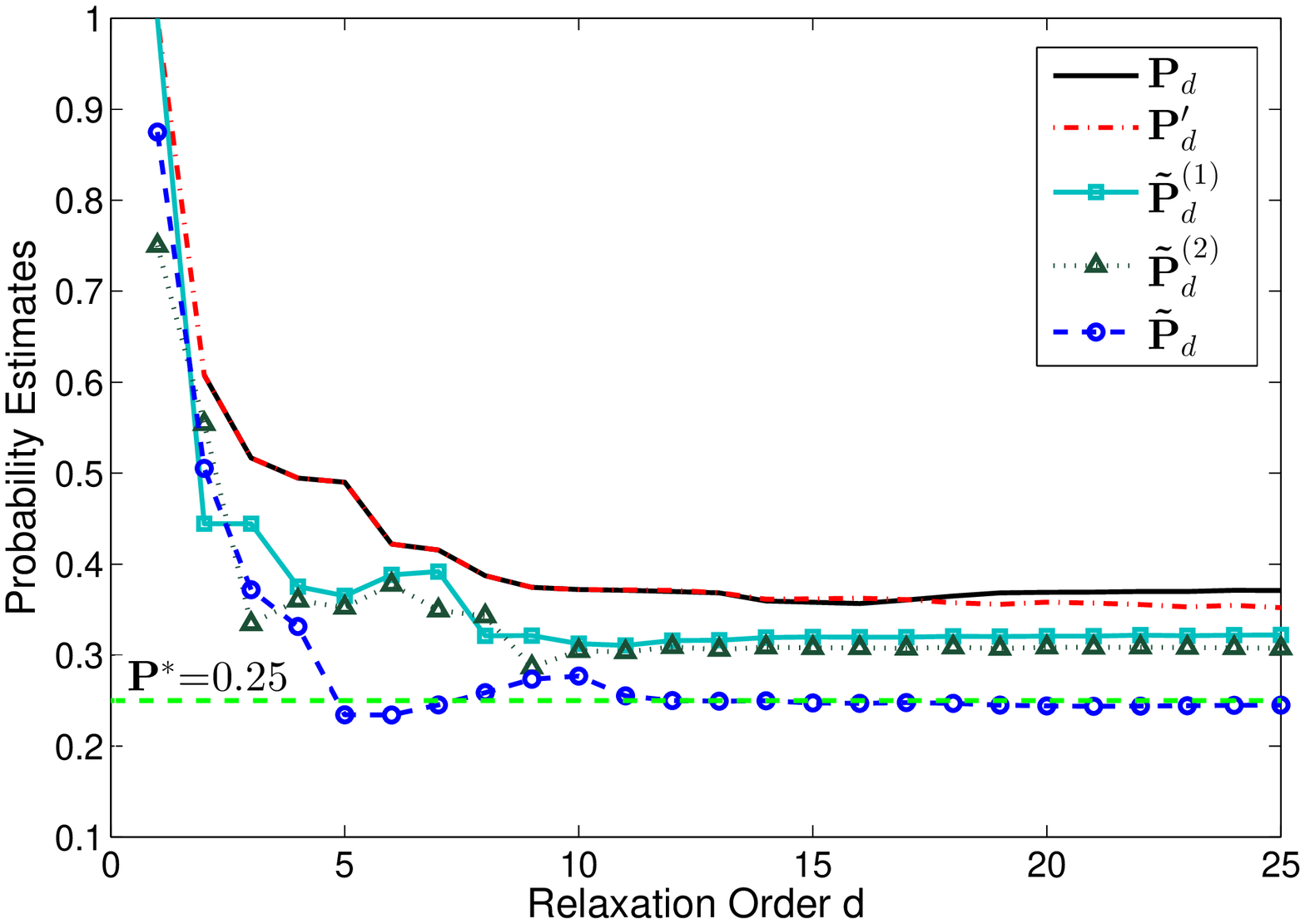}
\caption{$\mathbf{P}_d$, $\mathbf{P'}_d$, $\mathbf{\tilde{P}^{(1)}}_d$, $\mathbf{\tilde{P}^{(2)}}_d$, and $\mathbf{\tilde{P}}_d$ for increasing relaxation order $d$}
\label{fig:22}
\vspace{-0.25cm}
\end{figure}
\subsection{Orthogonal Basis}
In this paper, all polynomials are expanded in the usual monomial basis, and the SDPs are therefore formulated as optimization problems over ordinary monomial moments. However, one can  improve the numerical performance as in~\cite{r_vol} by employing an orthogonal basis of polynomials. First, we redefine the moment and localization matrices represented in the given orthogonal basis. Recall that the $S_{n,d}\times S_{n,d}$-moment matrix represented in monomial basis can be written as $M_d({\mathbf y}) = L_\mathbf{y}\left(\cB_d \cB_d^T\right)$, where $\cB_d^T=\left[x^{\alpha^{(1)}},\ldots, x^{\alpha^{(S_{n,d})}}\right]^T$ denotes the vector comprised of the elements of the monomial basis of $\mathbb R_{\rm d}[x]$, where $S_{n,d} := \binom{d+n}{n}$ and $\{\alpha^{(i)}\}_{i=1}^{S_{n,d}}=\mathbb{N} ^{\rm n}_d := \{\alpha \in \mathbb N^n : \norm{\alpha}_1 \leq d \}$ such that $\mathbf{0} = \alpha ^{(1)} <_g \ldots <_g \alpha ^{(S_{n,d})}$ are sorted in \emph{grevlex} order. Similarly, given a polynomial $\mathcal{P} \in \mathbb R [x]$ with coefficient vector $ \mathbf p = \{ p_{\gamma }\}_{\gamma\in\mathbb{N}^n}$ with respect to the monomial basis, its $S_{n,d}\times S_{n,d}$-localizing matrix represented in the monomial basis can be written as $M_d(\mathbf{y};\mathbf{p})=L_{\mathbf{y}}\left(\mathbf \cP \cB_d\cB_d^T\right)$.

Let $\{b_i\}_{i\in\mathbb{N}}$ be an orthogonal basis of \emph{univariate} polynomials on $[-1,1]$, i.e., $\int_{[-1,1]}b_i(t)b_j(t)~dt=0$ for all $i\neq j$. Without loss of generality, suppose that the degree of $b_i$ is equal to $i$ for all $i\in\mathbb{N}$. Given $n\geq 1$, for all $\alpha\in\mathbb{N}^n$, define $b_\alpha:\reals^n\rightarrow\reals$ such that $b_\alpha(x):=\prod_{i=1}^n b_{\alpha_i}(x_i)$, where $\alpha_i$ and $x_i$ are the $i$-th components of $\alpha\in\mathbb{N}^n$ and $x\in\reals^n$, respectively. Clearly $\{b_\alpha:\ \alpha\in\mathbb{N} ^{\rm n}_d\}$ is an orthogonal basis of multivariate polynomials on $[-1,1]^n$ with degree at most $d$, i.e., $\int_{[-1,1]^n}b_{\alpha^{(i)}}(x)~b_{\alpha^{(j)}}(x)~dx=0$ for all $1\leq i\neq j\leq S_{n,d}$. Let $\cB^o_d$ denote the vector of polynomials in $\mathbb R_{\rm d}[x]$ defined as ${\cB^o_d}^T=\left[b_{\alpha^{(1)}}(x),~b_{\alpha^{(2)}}(x),~\ldots,b_{\alpha^{(S_{n,d})}}(x)\right]$; and $T_d\in\reals^{S_{n,d}\times S_{n,d}}$ denote the one-to-one correspondence such that $\cB^o_d = T_d \cB_d$. Moreover, for a given sequence $\mathbf{y}=\{y_\alpha\}_{\alpha\in\mathbb{N}^n}$, let $L^o_\mathbf{y}:\reals[x]\rightarrow\reals$ be a linear map defined as
\begin{equation}
\label{eq:lin_ortho_map}
\cP \quad \mapsto \quad L^o_\mathbf{y}(\cP)=\sum_{\alpha\in\mathbb{N}^n}p^o_\alpha y_\alpha, \quad \hbox{where} \quad \cP(x)=\sum_{\alpha\in\mathbb{N}^n} p^o_\alpha b_\alpha(x).
\end{equation}
Given $y\in\reals^{S_{n,2d}}$ such that $y^T=\left[y_{\alpha^{(1)}},\ldots,y_{\alpha^{(S_{n,2d})}}\right]^T$, define its extension $\mathbf{y}=\{y_{\alpha}\}_{\alpha\in\mathbb{N}^n}$ such that $y_\alpha=0$ for all $\alpha\in\mathbb{N}^n$ with $\norm{\alpha}_1>2d$. For $\bar{y}:=T_{2d}^{-1}y$, define its extension $\mathbf{\bar{y}}$ similarly. Then for all $\cP\in\reals_d[x]$, we have $L^o_{\bf y}(\cP)=L_{\mathbf{\bar{y}}}(\cP)$. In the rest of the paper, we abuse the notation and write $\mathbf{\bar{y}}=T_{2d}^{-1}\mathbf{y}$. Then the moment matrix operator, $M^o_d(\mathbf y)$, for the given orthogonal basis is defined as
\begin{equation}
\label{eq:moment_orth}
M^o_d(\mathbf y) := L^o_{\mathbf y} \left(  \cB^o_d~{\cB^o_d}^T \right) = L_{T_{2d}^{-1}\mathbf y} \left( T_d\cB_d~{\cB_d}^T T_d^T \right)=T_d M_d\left(T_{2d}^{-1}\mathbf y\right)T_d^T.
\end{equation}
\newpage
For example for $d=2$ and $n=2$, the moment matrix under the orthogonal basis formed by Chebyshev polynomials of the first kind can be written as follows
\begin{equation} \label{moment matrix exa2}
\renewcommand{\arraystretch}{1.2}
M^o_2\left({\mathbf y}\right)=
\left[
         \begin{array}{cccccc}
           y_{00} & y_{10} & y_{01} & y_{20} & y_{11} & y_{02} \\
           y_{10} & \frac{y_{00}+y_{20}}{2} & y_{11} & \frac{y_{10}+y_{30}}{2} & \frac{y_{01}+y_{21}}{2} & y_{12} \\
           y_{01} & y_{11} & \frac{y_{00}+y_{02}}{2} & y_{21} & \frac{y_{10}+y_{12}}{2} & \frac{y_{01}+y_{03}}{2} \\
           y_{20} & \frac{y_{10}+y_{30}}{2} & y_{21} & \frac{y_{00}+y_{40}}{2} & \frac{y_{11}+y_{31}}{2} & y_{22} \\
           y_{11} & \frac{y_{01}+y_{21}}{2} & \frac{y_{10}+y_{12}}{2} & \frac{y_{11}+y_{31}}{2} & \frac{y_{00}+y_{20}+y_{02}+y_{22}}{4} & \frac{y_{11}+y_{13}}{2} \\
           y_{02} & y_{12} & \frac{y_{01}+y_{03}}{2} & y_{22} & \frac{y_{11}+y_{13}}{2} & \frac{y_{00}+y_{04}}{2} \\
         \end{array}
       \right].
\end{equation}

Let $\cP\in\reals[x]$ be a given polynomial with degree $\delta$, and $\textbf{p}=\{p_\alpha\}_{\alpha\in\mathbb{N}^n}$ denote its coefficient sequence  with respect to the standard monomial basis, i.e., $\cP(x)=\sum_{\alpha\in\mathbb{N}^n} p_\alpha x^\alpha$. For a given orthogonal basis, the localization matrix operator 
is defined as
\begin{equation}
\label{eq:local_orth}
M^o_d(\mathbf{y};\mathbf{p}) := L^o_{\mathbf y}\left( \cP \cB^o_d~{\cB^o_d}^T\right) = L_{T_{2d+\delta}^{-1}{\bf y}}\left( T_d \cP \cB_d~\cB_d^T T_d^T\right)=T_d M_d\left(T_{2d+\delta}^{-1}{\bf y};\mathbf{p}\right)T_d^T.
\end{equation}
Let $r:=\lceil\frac{\delta}{2}\rceil$. It is important to note that since $T_{2d}$ is invertible, $\{\mathbf{y}:\ M^o_d(\mathbf{y})\succeq 0,\ M^o_{d-r}(\mathbf{y};\mathbf{p})\succeq 0\}$ and $\{\mathbf{y}:\ M_d(\mathbf y)\succeq 0,\ M_{d-r}(\mathbf{y};\mathbf{p})\succeq 0\}$ are \emph{isomorphic}.
Hence, one can reformulate the SDP relaxation in \eqref{max_p4} using the new moment and localization matrix operators defined in \eqref{eq:moment_orth} and \eqref{eq:local_orth}, respectively; and the resulting problem stated in the given orthogonal basis is equivalent to \eqref{max_p4}. In order to illustrate the effect of orthogonal polynomial basis on the numerical behavior of the proposed method, we compared the two formulations of the simple example in~\eqref{max_exa1_def1}: the first formulation is given in \eqref{max_p4} using monomial basis, and the second formulation is obtained by replacing $M_d(.)$ and $M_{d-r_j}(.;\mathbf{p}_j)$ in \eqref{max_p4} with $M^o_d(.)$ and $M^o_{d-r_j}(.;\mathbf{p}_j)$, i.e., moment and localizing matrices in Chebyshev polynomial basis representations. In order to avoid matrix inversions as in \eqref{eq:moment_orth} and in \eqref{eq:local_orth}, we used Chebfun package~\cite{chebfun}, which can efficiently manipulate \emph{univariate} Chebyshev polynomials, to form $M^o_d(.)$ and $M^o_{d-r_j}(.;\mathbf{p}_j)$ that use \emph{multivariate} Chebyshev polynomials in a \emph{numerically stable} way; and solved the resulting SDP problems represented in the Chebyshev polynomial basis using SeDuMi. Figure~\ref{fig:2-3} shows that the approximations to the optimal probability $\mathbf{P}^*$ converge faster when Chebyshev polynomial basis is used as opposed to the standard monomial basis as relaxation order $d$ increases. For the problems in Chebyshev basis, the approximation $(x^o)^d$ to the optimal decision $x^*$ is formed similarly as $x^d$ -- see Section~\ref{sec:prob_approx}. For this example $x^d$ and $(x^o)^d$ sequences were close.
\begin{figure}[!h]
\centering
\includegraphics[scale=0.5]{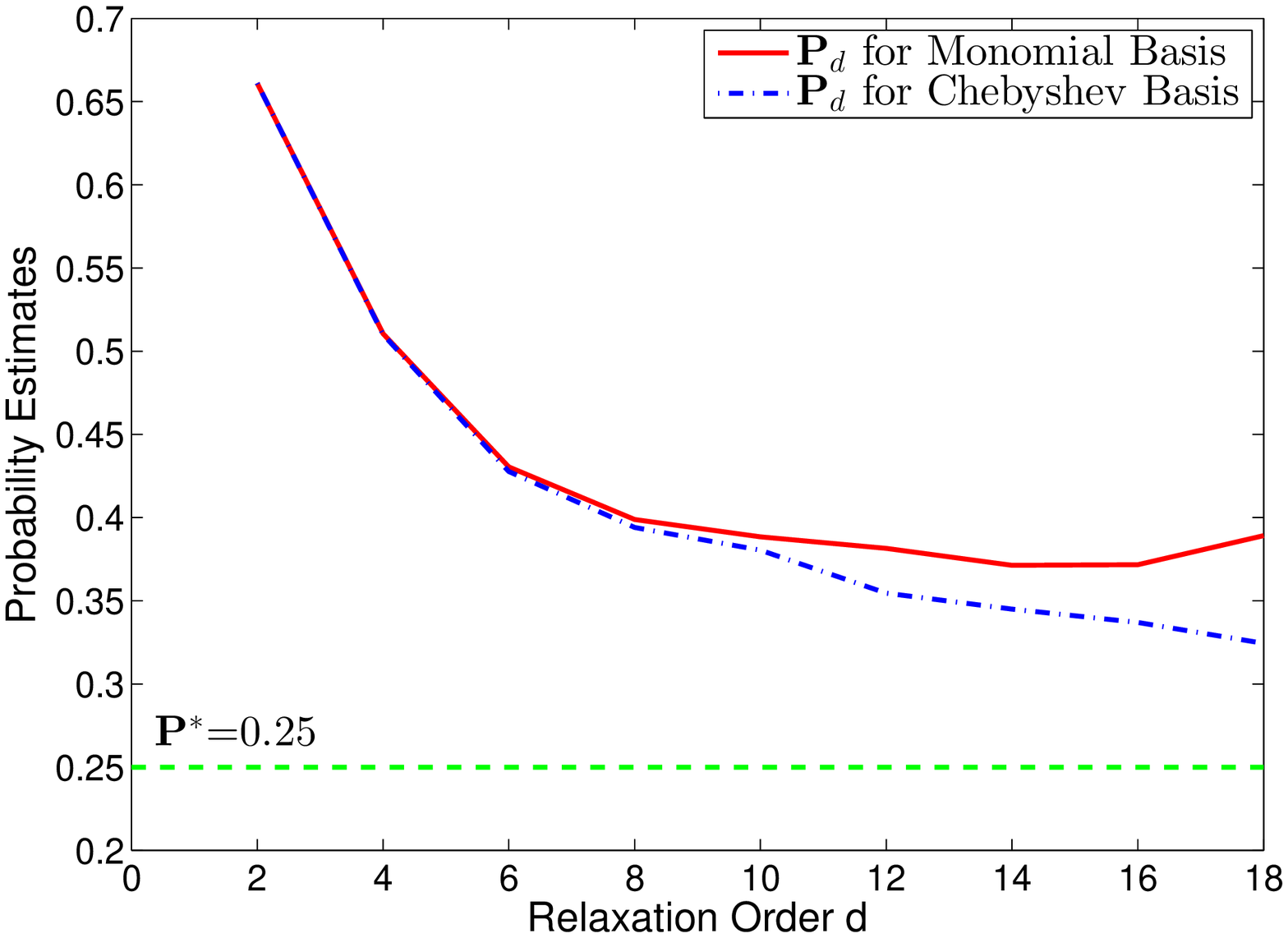}
\caption{$\mathbf{P}_d$ for monomial and Chebyshev polynomial bases}
\label{fig:2-3}       
\end{figure}
\section{Chance Optimization over a Union of Sets}
\label{sec:union_prob}
We now focus on the more general setting of the chance optimization problem in \eqref{intro_union}. Given polynomials $\mathcal{P}^{k}_j:\reals^n\times\reals^m\rightarrow\reals$ with degree $\delta^{(k)}_j$ for $j=1,\ldots,\ell_k$ and $k=1,\ldots,N$,  the semi-algebraic set of interest is $\cK=\cup_{k=1}^N\cK_k$, where
\begin{equation} \label{min_set}
\mathcal K_k =\left\{(x,q)\in\reals^n\times\reals^m: \mathcal{P}_j^{(k)}(x,q)\geq 0,\ j=1,\ldots,\ell_k\right\},\quad k=1,\ldots,N.
\end{equation}

Similar to the previous section, we need \emph{Putinar's property} to hold for $\cK_k$ for all $k=1,\ldots,N$. With the following assumption, we can ensure this.
\begin{assumption}
\label{assump:putinar2}
$\mathcal K=\cup_{k=1}^N\cK_k$ is bounded, where $\cK_k$ is defined in \eqref{min_set}.
\end{assumption}

Hence, as discussed in Remark~\ref{rem:prob_space}, we can assume without loss of generality that $\cK\subseteq \chi \times \mathcal Q$ and the probability measure $\mu_q\in\cM(\cQ)$, where $\chi =[-1,1]^n$ and $\cQ =[-1,1]^m$. Therefore, for all $(x,q)\in\cK$, we have $\norm{x}_2^2+\norm{q}_2^2\leq m+n$. Define $\cP^{(k)}_0(x,q):=m+n-\sum_{i=1}^nx_i^2-\sum_{i=1}^m q_i^2$ for all $k=1,\ldots, N$. $\cK_k$ can be represented as $\cK_k =\left\{(x,q): \mathcal{P}_j^{(k)}(x,q)\geq 0,\ j=0,\ldots,\ell_k\right\}$ --note that index $j$ starts from $0$. Since polynomials are continuous in $(x,q)$, the new representation of $\cK_k$ satisfies \emph{Putinar's property} for each $k$ and we still have $\cK=\cup_{k=1}^N\cK_k$.

The objective of this section is to provide a sequence of SDP relaxations to the chance optimization problem in \eqref{intro_union} with $N>1$,
 and show that the results presented in the previous sections can be easily extended for this case. More precisely, we start by providing an equivalent problem in the measure space and then develop relaxations based on moments of measures.

\subsection{ An Equivalent Problem}
As an intermediate step in the development of convex relaxations of \eqref{intro_union}, an equivalent problem in the measure space is provided below.\vspace{-0.3cm}
\begin{align}
\mathbf{P_{\mu_q}^*}:=&\ \sup_{\mu_k,~\mu_x} \sum_{k=1}^{N} \int d\mu_k, \label{min_p2} \\
\hbox{s.t.}\quad & \sum_{k=1}^{N} \mu_k \preccurlyeq \mu_x \times \mu_q, \label{eq:y_cons_min_p2}\subeqn\\
&\mu_x \hbox{ is a probability measure}, \label{eq:x_cons_min_p2}\subeqn\\
&\mu_x\in \cM(\chi),\quad \mu_k\in\cM(\cK_k)\quad k=1,\ldots,N. \label{eq:supp_min_p2}\subeqn
\end{align}
This problem is equivalent to the problem addressed in this paper in the following sense.
\begin{theorem} \label{Min Theo 1}
The optimization problems in \eqref{intro_union} and \eqref{min_p2} are equivalent in the following sense:
\begin{enumerate}[i)]
\item The optimal values are the same, i.e. $\mathbf{P^*}=\mathbf{P_{\mu_q}^*}$.
\item If an optimal solution to \eqref{min_p2} exists, call it $\mu^*_x$, then any $x^*\in supp(\mu^*_x)$ is an optimal solution to \eqref{intro_union}.
\item If an optimal solution to \eqref{intro_union} exists, call it $x^*$, then Dirac measure at $x^*$, $\mu_x = \delta_{x^*}$ and $\mu = \delta_{x^*} \times \mu_q$ is an optimal solution to \eqref{min_p2}.
\end{enumerate}
\end{theorem}
\begin{proof}
Let $\mathbf{P^*}$ denote the optimal value of \eqref{intro_union}, and $\cK=\cup_{k=1}^N\cK_k$, where $\cK_k$ is defined in \eqref{min_set}. It can be proven as in Theorem~\ref{Max Theo 1} that
\begin{equation} \label{eq:product_union}
\mathbf{P^*}=\sup_{\mu_x\in\cM(\chi)} \sup_{\mu\in\cM(\cK)} \int d\mu \quad \hbox{s.t.}\quad \mu \preccurlyeq \mu_x \times \mu_q,\ \mu_x(\chi)=1.
\end{equation}
Let $\{\mu_k\}_{k=1}^N$ and $\mu_x$ be a feasible solution to \eqref{min_p2} with objective value $P$. Since $\mu_k\in\cM(\cK_k)\subset\cM(\cK)$ for all $k=1,\ldots,N$, we have $\sum_{k=1}^N\mu_k\in\cM(\cK)$. Hence, $\left(\sum_{k=1}^N\mu_k,~\mu_x\right)$ is a feasible solution to \eqref{eq:product_union} with objective value $P$, as well. Clearly, this shows that $\mathbf{P_{\mu_q}^*}\leq \mathbf{P^*}$, where $\mathbf{P_{\mu_q}^*}$ denotes the optimal value of \eqref{min_p2}.

Suppose that $(\mu,~\mu_x)$ is a feasible solution to \eqref{eq:product_union} with objective value $P$. Define $\{\mu_k\}_{k=1}^N$ as follows
\begin{align}
\label{eq:muk_dfn}
\mu_k(S):=\mu\left(S\cap\left(\cK_k\setminus\bigcup_{j=0}^{k-1}\cK_j\right)\right),\quad \forall S\in\Sigma(\cK),
\end{align}
for all $k=1,\ldots,N$, where $\cK_0:=\emptyset$ and $\Sigma(\cK)$ denotes the Borel $\sigma$-algebra over $\cK$. Definition in \eqref{eq:muk_dfn} implies that $\mu_k\in\cM(\cK_k)$ for all $k=1,\ldots,N$, and $\sum_{k=1}^N\mu_k(S)=\mu(S)$ for all $S\in\Sigma(\cK)$. Hence, $\{\mu_k\}_{k=1}^N$ and $\mu_x$ form a feasible solution to \eqref{max_p2} with objective value equal to $P$. Therefore, $\mathbf{P_{\mu_q}^*}=\mathbf{P^*}$.
\end{proof}
\vspace{-0.3cm}
\subsection{Semidefinite Relaxations}
In this section, a sequence of semidefinite programs is provided which can arbitrarily approximate the optimal solution of \eqref{min_p2}. As before, this is done by considering moments of measures instead of the measures themselves.
Define the following optimization problem indexed by the relaxation order $d$.

\begin{align}
\mathbf{P_d}:= &\sup_{\mathbf{y}_k\in\reals^{S_{n+m,2d}},\ \mathbf{y_x}\in\reals^{S_{n,2d}}} \sum_{k=1}^{N} \left(\mathbf{y}_k\right)_\mathbf{0}, \label{min_p3}\\
\hbox{s.t.}\quad & M_d({\mathbf y_k})\succcurlyeq 0,\ M_{d-r^{(k)}_j}\left(\mathbf{y}_k;\mathbf{p}^{(k)}_j\right)\succcurlyeq 0 , \quad j=1,\ldots ,l_k,\quad k=1,\ldots,N \label{eq:y_cons_i_min_p3}\subeqn\\
&M_d(\mathbf y_\mathbf x)\succcurlyeq 0,\ \Vert \mathbf y_\mathbf x \Vert_\infty \leq 1,\ \left(\mathbf{y_x}\right)_\mathbf{0}=1, \label{eq:x_cons_i_min_p3}\subeqn\\
&M_d\left(A_d\mathbf y_\mathbf x-\sum_{k=1}^{N}{\mathbf y_k}\right)\succcurlyeq 0, \label{eq:moment_ineq_i_min_p3}\subeqn
\end{align}
where $\delta^{(k)}_j$ is the degree of $\cP^{(k)}_j$, $r^{(k)}_j:=\left\lceil\frac{\delta^{(k)}_j}{2}\right\rceil$ for all $1\leq j\leq \ell_k$ and $1\leq k\leq N$; and $A_d:\reals^{S_{n,2d}}\rightarrow\reals^{S_{n+m,2d}}$ is defined similarly to $\mathbf{A}$ in \eqref{max_p3}. Indeed, let $\mathbf{y_q}:=\{y_{q_\beta}\}_{\beta\in\mathbb{N}_{2d}^m}$ be the truncated moment sequence of $\mu_q$. Then for any given $\mathbf{y_x}=\{y_{x_\alpha}\}_{\alpha\in\mathbb{N}_{2d}^n}$, $\mathbf{y}=A_d\mathbf{y_x}$ such that $y_\theta=y_{q_\beta}y_{x_\alpha}$ for all $\theta=(\beta,\alpha)\in\mathbb{N}_{2d}^{n+m}$.


Next, we show that the sequence of optimal solutions to the SDPs in \eqref{min_p3} converges to the solution of the infinite dimensional SDP in \eqref{min_p2}. More precisely, we have the following result.
\begin{theorem} \label{Max theo 3}
For all $d\geq 1$, there exists an optimal solution $\left(\{\mathbf{y}_k^d\}_{k=1}^N,\mathbf{y}^d_\mathbf{x}\right)$ to \eqref{min_p3} with the optimal value $\mathbf{P}_d$. Moreover,
\begin{enumerate}[i)]
\item $\lim_{d\in\integers_+}\mathbf{P}_d=\mathbf{P^*}$, the optimal value of \eqref{intro_union}.
\item Let $\cS:=\left\{\left(\{\mathbf{y}_k^d\}_{k=1}^N,\mathbf{y}^d_\mathbf{x}\right)\right\}_{d\in\integers_+}$ such that each element is obtained by zero-padding $\mathbf{y}^d$ and $\mathbf{y}_k^d$ for $1\leq k\leq N$. There exists an accumulation point of $\cS$ in the weak-$\star$ topology of $\ell_\infty$, and for every accumulation point of $\cS$, there exists corresponding representing measures $\left(\{\mu_k^*\}_{k=1}^N,\mu^*_x\right)$ that is optimal to \eqref{min_p2} and any $x^*\in supp(\mu_x^*)$ is optimal to \eqref{intro_union}.
\end{enumerate}
\end{theorem}
\begin{proof}
Let $\{\mathbf{y}_k\}_{k=1}^N\subset\reals^{S_{n+m,2d}}$ and $\mathbf{y}_\mathbf{x}\in\reals^{S_{n,2d}}$ be a feasible solution to \eqref{min_p3}. As in Theorem~\ref{Max theo 2}, it can be shown that
\begin{equation}
\label{eq:moment_bound_union}
\max \left\{ \left(\mathbf{y}\right)_\mathbf{0}, \max_{i=1,\ldots,n+m} L_{\mathbf{y}}\left(x_i^{2d}\right) \right\}\leq 1,
\end{equation}
where $\mathbf{y}:=\sum_{k=1}^N\mathbf{y}_k$. Note that $L_{\mathbf{y}}\left(x_i^{2d}\right)=\sum_{k=1}^N L_{\mathbf{y}_k}\left(x_i^{2d}\right)$, and $\{L_{\mathbf{y}_k}\left(x_i^{2d}\right)\}_{i=1}^{n+m}$ is a subset of diagonal elements of $M_d({\mathbf y_k})\succeq\mathbf{0}$ for each $k\in\{1,\ldots,N\}$. Hence, $L_{\mathbf{y}_k}\left(x_i^{2d}\right)\geq 0$ for all $i\in\{1,\ldots,n+m\}$ and $k\in\{1,\ldots,N\}$. Therefore, \eqref{eq:moment_bound_union} implies that $\max \left\{ \left(\mathbf{y}_k\right)_\mathbf{0}, \max_{i=1,\ldots,n+m} L_{\mathbf{y}_k}\left(x_i^{2d}\right) \right\}\leq 1$ for all $k\in\{1,\ldots,N\}$. Lemma~\ref{mom_bound 1} implies that $|(y_k)_\alpha|\leq 1$ for all $\alpha\in\mathbb{N}^{n+m}_{2d}$. Therefore, the feasible region is bounded. The rest of the proof is exactly the same as in Theorem~\ref{Max theo 2}.
\end{proof}

\section{Implementation and Numerical Results}
In previous sections, we showed that chance optimization problem in \eqref{intro_union} can be relaxed to a sequence of SDPs. In this section, we go one step further to improve approximation quality of the relaxed problems in practice and implement an efficient first-order algorithm to solve the resulting SDP relaxations.
\subsection{Regularized Chance Optimization Using Trace Norm}
As shown in Theorem~\ref{Max Theo 1} and Theorem~\ref{Min Theo 1}, if the chance optimization problems in \eqref{intro_max} and \eqref{intro_union} have unique optimal solution $x^*$, then the optimal distribution $\mu^*_x$ is a Dirac measure whose mass is concentrated on the single point $x^*$, i.e., its support is the singleton $\lbrace x^* \rbrace$. Such distributions, have moment matrices with rank one. To improve the solution quality of the algorithm, one can incorporate this observation in the formulation of the relaxed problem. For the sake of notational simplicity, in this section we will consider the regularized version of chance optimization problem \eqref{max_p4} for presenting the algorithm: 
\begin{equation} \label{rank problem}
 \min_{\mathbf{y}\in\reals^{S_{n+m,2d}},\ \mathbf{y_x}\in\reals^{S_{n,2d}}}   \omega_r \Tr( M_d(\mathbf{y_x}) )  -  (\mathbf{y})_\mathbf{0}\quad \hbox{subject to}\quad \eqref{eq:y_cons_i},~\eqref{eq:x_cons_i},~\eqref{eq:moment_ineq_i}
\end{equation}
for some $\omega_r > 0$, where $\Tr(.)$ denotes the trace function. Our objective is to achieve the maximum probability with a low-rank moment matrix $ M_d(\mathbf{y}^*_\mathbf{x})$, hopefully with rank 1. To this end, we regularize the objective with trace norm. Since $ M_d(\mathbf{y^*_x}) \succcurlyeq 0$, $ \Tr(M_d(\mathbf{y}^*_\mathbf{x}))$ is equal to sum of singular values of $ M_d(\mathbf{y}^*_\mathbf{x})$, which is called the nuclear norm of $M_d(\mathbf{y}^*_\mathbf{x})$. This is a well known approach for obtaining low-rank solutions. Indeed, the nuclear norm is the convex envelope of the rank function and, in practice, produces good results; see \cite{r_nuc norm 1} and \cite{r_nuc norm 2} for details.

To be able to solve the SDP in \eqref{rank problem} involving large scale matrices in practice, one need to implement an efficient convex optimization algorithm. Recently, a first-order augmented Lagrangian algorithm ALCC has been proposed in \cite{r_Allc} to deal with regularized conic convex problems. We will adapt this algorithm to solve SDPs of the form in \eqref{rank problem}. In the following section, we briefly discuss the algorithm ALCC.
\subsection{First-Order Augmented Lagrangian Algorithm}
\label{sec:ALCC}
Consider the optimization problem:
\begin{equation}
\label{eq:composite_problem}
(P):\ p^*=\min\{\rho(x)+\gamma(x):\ \ A(x)-b\in\cC\},
\end{equation}
where $\gamma:\reals^n\rightarrow\reals$ is a convex function such that $\grad \gamma$ is Lipschitz continuous with constant $L_\gamma$, $\rho:\reals^n\rightarrow\reals\cup\{+\infty\}$ is a closed convex function such that $\Delta:=\dom(\rho)$ is convex compact set, $A:\reals^n\rightarrow\reals^m$ is a \emph{linear map}, and $\cC\subset\reals^m$ is a closed convex cone. Let $\cC^*:=\lbrace  \theta\in \mathbb{R}^n: \langle z,\theta \rangle \geq 0,\ \forall z\in C \rbrace$ denote the dual cone of $\cC$, and $B>0$ denote the diameter of $\Delta$, i.e., $B=\max\{\norm{x-y}_2:\ x,y\in\Delta\}$; and we assume that $B$ is given. Given a penalty parameter $\nu>0$ and Lagrangian dual multiplier $\theta\in\cC^*$, the augmented Lagrangian for (P) in \eqref{eq:composite_problem} is given by
\begin{equation}
\cL(x;\nu,\theta):=\tfrac{1}{\nu}\left(\rho(x)+\gamma(x)\right)+\tfrac{1}{2}d_{\cC}(A(x)-b-\theta)^2,
\end{equation}
where $d_{\cC}:\reals^m\rightarrow\reals$ denotes the distance function to cone $\cC$, i.e., $d_{\cC}(\bar{z}):=\norm{\bar{z}-\Pi_{\cC}(\bar{z})}_2$, and $\Pi_{\cC}(\bar{z}):=\argmin\{\norm{z-\bar{z}}_2:\ z\in\cC\}$ denotes the Euclidean projection of $\bar{z}$ onto $\cC$. Given $\nu_k>0$ and $\theta_k\in\cC^*$, we define $\cL_k(x):=\cL(x;\nu_k,\theta_k)$ and $\cL_k^*:=\min_x \cL_k(x)$. Let $f_k:\reals^n\rightarrow\reals$ such that $f_k(x):=\tfrac{1}{\nu_k}\gamma(x)+\tfrac{1}{2}d_{\cC}(A(x)-b-\theta)^2$; hence, $\cL_k^*=\min_x \tfrac{1}{\nu_k}\rho(x)+f_k(x)$. It is important to note that $f_k$ is a convex function with Lipschitz continuous gradient $\grad f_k(x)=\tfrac{1}{\nu_k}\gamma(x)-A^*\left(\Pi_{\cC^*} (\theta_k+b- A(x))\right)$; and the Lipschitz constant of $\grad f_k$ is equal to $L_k := \frac{1}{\nu_k} L_\gamma + \sigma_{\max}^2(A)$, where $A^*: \mathbb{R}^{m} \rightarrow \mathbb{R}^n $ denotes the adjoint operator of $A: \mathbb{R}^{n} \rightarrow \mathbb{R}^m$, and $\sigma_{\max}(A)$ denotes the maximum singular value of the linear map $A$. Therefore, given $\epsilon_k>0$, an $\epsilon_k$-optimal solution, $\tilde{x}_k$, to $\cL_k^*:=\min_x \cL_k(x)$ can be efficiently computed such that $\cL_k(\tilde{x}_k)-\cL_k^*\leq\epsilon_k$ using an Accelerated Proximal Gradient~(APG) algorithm~\cite{Beck09_1J,Nesterov05,Nesterov04,Tseng08} within $\ell_k^{\max}(\epsilon_k):=B~\sqrt{\frac{2L_k}{\epsilon_k}}$ APG iterations. In each APG iteration, $\grad f_k$, $\Pi_{\cC^*}$ and proximal map of $\rho$ are all evaluated \emph{once}.

ALCC algorithm proposed in~\cite{r_Allc} can generate a minimizing sequence $\{x_k\}$ to (P) in~\eqref{eq:composite_problem} by \emph{inexactly} solving a sequence of subproblems $\min_x \cL_k(x)$. In particular, given inexact computation parameters $\alpha_k>0$ and $\eta_k>0$, $x_k$ is computed such that either one of the following conditions holds:
\begin{align}
\cL_k(x_k)-\cL_k^*\leq\tfrac{\alpha_k}{\nu_k}, \label{eq:eps-cond}\\
\exists s_k\in\partial \cL_k(x_k) \quad \hbox{ such that } \quad \norm{s_k}_2\leq \tfrac{\eta_k}{\nu_k}, \label{eq:subgrad-cond}
\end{align}
where $\partial\cL_k(x_k)$ denotes the subdifferential of $\cL_k$ at $x_k$ -- the inexact optimality criteria in \eqref{eq:eps-cond} and \eqref{eq:subgrad-cond} have been successfully implemented in other first-order augmented Lagrangian algorithms in \cite{aybat12,aybat12_fal,aybat15_fal} as well. Then dual Lagrangian multiplier is updated: $\theta_{k+1}=\tfrac{\nu_k}{\nu_{k+1}}\Pi_{C^*}(\theta_k+b-A(x_k))$. For given $c,\beta>1$, fix the parameter sequence as follows: $\nu_k=\beta^k\nu_0$, $\alpha_k=\frac{1}{k^{2(1+c)}\beta^k}\alpha_0$, and $\eta_k=\frac{1}{k^{2(1+c)}\beta^k}\eta_0$ for all $k\geq 1$; and let $\{x_k,\theta_k\}\subset\Delta\times C^*$ be the primal-dual ALCC iterate sequence. \textbf{Theorem~3.10} in~\cite{r_Allc} shows that $\lim_k \theta_k\nu_k$ exists and it is an optimal solution to the dual problem. Moreover, \textbf{Theorem~3.8} shows that for all $\epsilon>0$, $x_k$ is $\epsilon$-feasible, i.e., $d_{\cC}(Ax_k-b)\leq\epsilon$, and $\epsilon$-optimal, i.e., $|\rho(x_k)+\gamma(x_k)-p^*|\leq\epsilon$ within $\log(1/\epsilon)$ ALCC iterations, i.e., $k=\cO(\log(1/\epsilon))$, which requires $\cO(\epsilon^{-1}\log(\epsilon^{-1}))$ 
APG iterations in total. Moreover, every limit point of $\{x_k\}$ is optimal (when $A\in\reals^{m\times n}$ is surjective, the techniques used for proving \textbf{Theorem~4} in \cite{aybat12} can be used to improve the rate result to $\cO(1/\epsilon)$).

Now consider the following problem $p^*=\min_{x\in\Delta}\{\gamma(x):\ A(x)-b\in\cC\}$, where $\Delta\subset\reals^n$ is a compact convex set. Note that 
this problem can be written as a special case of \eqref{eq:composite_problem} by setting $\rho(x)=\mathbf{1}_\Delta(x)$, the indicator function of the set $\Delta$, i.e., $\mathbf{1}_\Delta(x)=0$, if $x\in\Delta$, and equal to $+\infty$, if $x\not\in\Delta$. In Figure~\ref{tab:1}, we present the ALCC algorithm customized to solve $p^*=\min_{x\in\Delta}\{\gamma(x):\ A(x)-b\in\cC\}$. Note that Step~\ref{algeq:grad_comp} and Step~\ref{algeq:proj_grad} in Figure~\ref{tab:1} are the bottleneck steps (one $\grad\gamma$ evaluation and two projections: one onto $\cC^*$, and one onto $\Delta$) -- in Step~\ref{algeq:grad_comp} $\grad f_k$ is evaluated at $x^{(2)}_\ell$, and then in Step~\ref{algeq:proj_grad} $x^{(1)}_\ell$ is computed via a projected gradient step of length $1/L_k$. In this customized version, ALCC iterate $x_k$ is set to $x_\ell^{(1)}$ whenever either $\ell>\ell_k^{\max}$ or $\norm{x_\ell^{(1)}-x_\ell^{(2)}}_2\leq\tfrac{\eta_k}{\nu_k}$. Note that $\ell_k^{\max}:=k^{1+c}\beta^k B\sqrt{\frac{2\nu_0 L_k}{\alpha_0}}$, which is equal to $\ell_k^{\max}(\epsilon_k)$ when $\epsilon_k=\tfrac{\alpha_k}{\nu_k}$. Therefore, if $\ell>\ell_k^{\max}$, then $\cL_k(x_k)-\cL_k^*\leq\tfrac{\alpha_k}{\nu_k}$ -- this follows from the complexity of Accelerated Proximal Gradient algorithm (lines 9-19 in Figure~5.1) running on $\min \cL_k(x)$; next we'll show that if $\norm{x_\ell^{(1)}-x_\ell^{(2)}}_2\leq\tfrac{1}{2L_k}\tfrac{\eta_k}{\nu_k}$, then \eqref{eq:subgrad-cond} holds. For $\rho(x)=\mathbf{1}_\Delta(x)$, we have $\cL_k(x)=\rho(x)+f_k(x)$. Suppose that for some $\ell$, $\norm{x_\ell^{(1)}-x_\ell^{(2)}}_2\leq\tfrac{1}{2L_k}\tfrac{\eta_k}{\nu_k}$ holds. Note that $g_\ell$ computed in Line 11 is equal to $\grad f_k(x_\ell^{(2)})$; thus $x_\ell^{(1)}$ computed in Line~12 is equal to $\Pi_\Delta(x_\ell^{(2)}-\grad f_k(x_\ell^{(2)})/L_k)$, where $L_k:=\tfrac{1}{\nu_k}L_\gamma+\sigma_{\max}^2(A)$ is the Lipschitz constant of $\grad f_k$. One can easily show that $x_\ell^{(2)}-\grad f_k(x_\ell^{(2)})/L_k-x_\ell^{(1)}\in\partial\rho(x_\ell^{(1)})$; and since $\rho$ is the indicator function, we also have $L_k\left(x_\ell^{(2)}-x_\ell^{(1)}\right)-\grad f_k(x_\ell^{(2)})\in\partial\rho(x_\ell^{(1)})$. Hence, $s_k:=L_k\left(x_\ell^{(2)}-x_\ell^{(1)}\right)+\grad f_k(x_\ell^{(1)})-\grad f_k(x_\ell^{(2)})\in \partial P_k(x_\ell^{(1)})$. Since $\grad f_k$ is Lipschitz continuous, we have $\norm{\grad f_k(x_\ell^{(1)})-\grad f_k(x_\ell^{(2)})}_2\leq L_k\norm{x_\ell^{(2)}-x_\ell^{(1)}}_2$. Therefore, we have $\norm{s_k}_2\leq2L_k\norm{x_\ell^{(2)}-x_\ell^{(1)}}_2\leq\frac{\eta_k}{\nu_k}$.

\begin{figure}[!h]
    \rule[0in]{6.5in}{1pt}\\
    \textbf{Algorithm ALCC}$~(x_0,\nu_0,\alpha_0,L_\gamma, B)$\\
    \rule[0.125in]{6.5in}{0.1mm}
    \vspace{-0.35in}
    {\small
    \begin{algorithmic}[1]
    \STATE $k\gets1$, $\theta_1\gets\mathbf{0}$
    \STATE $\eta_0\gets0.5~\norm{\grad\gamma\left(x_0\right)-\nu_0 A^{*}\left(\Pi_{C^*}\left(b-A(x_0)\right)\right)}_2$
    \WHILE{$k\geq 1$} \label{algeq:stop_ALCC}
    \STATE $\ell\gets 0$, $t_1\gets 1$,
    \STATE $x_0^{(1)}\gets x_{k-1}$, $x_1^{(2)} \gets x_{k-1}$
    \STATE $L_k\gets\frac{1}{\nu_k}L_\gamma+\sigma_{\max}^2(A)$,\ $\ell_k^{\max}\gets k^{1+c}\beta^k B\sqrt{\frac{2\nu_0 L_k}{\alpha_0}}$
    \STATE $\nu_k \gets \beta^k \nu_0$,\ $ \alpha_k \gets \dfrac{1}{k^{2(1+c)}\beta^k}\alpha_0$,\ $ \eta_k \gets \dfrac{1}{k^{2(1+c)}\beta^k}\eta_0$
    \STATE $\mathrm{STOP}\gets\mathbf{false}$
    \WHILE{ $\mathrm{STOP}=\mathbf{false}$} \label{algeq:stop}
    \STATE $\ell \gets \ell + 1$
    \STATE $g_\ell \gets \frac{1}{\nu_k}\grad\gamma\left(x_\ell^{(2)}\right)- A^{*}\left(\Pi_{C^*}\left(\theta_k+b-A\left(x_\ell^{(2)}\right)\right)\right) $ \label{algeq:grad_comp}
    \STATE $x_\ell^{(1)} \leftarrow \Pi_\Delta\left(x_\ell^{(2)}-g_\ell/L_k\right)$ \label{algeq:proj_grad}
    \IF {$\norm{x_\ell^{(1)} - x_\ell^{(2)}}_2 \leq \dfrac{1}{2L_k}\dfrac{\eta_k}{\nu_k}$ \textbf{or} $\ell>\ell_k^{\max}$}
        \STATE $\mathrm{STOP}\gets\mathbf{true}$
        \STATE $x_k\gets x_\ell^{(1)}$
    \ENDIF
    \STATE $t_{\ell+1}\gets \left(1+\sqrt{1+4~t^2_{\ell}}\right)/2$
    \STATE $x_{\ell+1}^{(2)} \gets x_\ell^{(1)} +
    \left(\frac{t_{\ell}-1}{t_{\ell+1}}\right)\left(x_\ell^{(1)}-x_{\ell-1}^{(1)}\right)$
    \ENDWHILE
    \STATE $\theta_{k+1} \leftarrow \frac{\nu_k}{\nu_{k+1}} \Pi_{C^*}(\theta_k+b-A(x_k))$
    \ENDWHILE
    \end{algorithmic}
    \rule[0.25in]{6.5in}{0.1mm}
    }
    \vspace{-0.5in}
    \caption{first-order Augmented Lagrangian algorithm for Conic Convex~(ALCC) problems}\label{tab:1}
    \vspace{-0.35cm}
\end{figure}
Semidefinite program of \eqref{rank problem} is a special case of the conic convex problem in~\eqref{eq:composite_problem}, where $\gamma(\mathbf{y_x,y})=c_r^T\mathbf{y_x}+c_p^T\mathbf{y}$ for some $c_r\in\reals^{S_{n,2d}}$ and $c_p\in\reals^{S_{n+m,2d}}$ since the objective of \eqref{rank problem} is linear in $(\mathbf{y, y_x})$; hence, $L_\gamma=0$, the conic constraint $ A(.)-b \in C$ in~\eqref{eq:composite_problem} is a linear matrix inequality~(LMI), with $\cC=\cC^*$ being the cone of positive semidefinite matrices $\mathbb{S}_+$, and the compact set $\Delta=\{(\mathbf{y},\mathbf{y_x}):\ \norm{\mathbf{y}}_\infty\leq 1,\ \norm{\mathbf{y_x}}_\infty\leq 1,\ (\mathbf{y_x})_0=1\}$. Hence, $\Pi_\cC(.)=\Pi_{\cC^*}(.)$ can be computed using one eigenvalue decomposition, and $\Pi_\Delta(.)$ is very efficient and can be computed in linear time. In our numerical experiments in Section~\ref{sec:numerical}, we used $\norm{x_k-x_{k-1}}_2/(1+\norm{x_{k-1}}_2)\leq \mathrm{tol}$ as the stopping condition for ALCC.
\newpage
\subsection{Numerical Examples}
\label{sec:numerical}
In this section, four numerical examples are presented that illustrate the performance of the proposed methodology, discussed in Sections \ref{sec:intersection_prob} and \ref{sec:union_prob}. We compared the augmented Lagrangian algorithm, ALCC, presented in Section~\ref{sec:ALCC} with GloptiPoly, which is a Matlab-based toolbox aimed at optimizing moments of measures \cite{r_GloptiPoly}, to compute approximate solutions to the chance constrained problems in \eqref{intro_union} and \eqref{intro_max}. 
In all the tables, for problems of the form \eqref{intro_max}, i.e., $N=1$, $\mathbf{P}_d$, $\mathbf{P'}_d$, $\mathbf{\bar{P}}_d$, and $\mathbf{\tilde{P}}_d$ denote the optimal probability estimates defined similarly as in Section~\ref{sec:prob_approx} for $x^d$ obtained by solving the regularized problem in \eqref{rank problem}; 
for problems of the form \eqref{intro_union}, i.e., $N>1$, these estimates can be defined naturally using $(\mathbf{y}^d,\mathbf{y}^d_\mathbf{x})$ with $\mathbf{y}^d:=\sum_{k=1}^N\mathbf{y}^d_k$; and $d\in\integers_+$ denotes the relaxation order. In order to compute $\mathbf{P^*}$ and $\mathbf{\bar{P}}_d$, we used Monte Carlo simulation discussed in Section~\ref{sec:simulation}. In all the tables, $\mathbf{iter}$ denotes the total number of algorithm iterations, and $\mathbf{cpu}$ denotes the computing time in \emph{seconds} required for computing $\mathbf{P}_d$; $\mathbf{n_{\mathrm{var}}}$ denotes the number of variables, i.e., total number of moments used. For ALCC $\mathbf{iter}$ is the total number of APG iterations, and for GloptiPoly it denotes the total number of SeDuMi~\cite{sedumi} iterations.
\subsubsection{Monte Carlo Simulation}
\label{sec:simulation}
To test the accuracy of the results obtained using ALCC and GloptiPoly, we used Monte Carlo integration to estimate an optimal solution and the corresponding optimal probability. Let $\cK\subset\reals^n\times\reals^m$ be the given semialgebraic set such that $\Pi_1:=\{x\in\reals^n:\ \exists q\in\reals^m \hbox{ s.t. } (x,q)\in\cK\}\subset\chi:=[-1,1]^n$, and $\Pi_2:=\{q\in\reals^m:\ \exists q\in\reals^m \hbox{ s.t. } (x,q)\in\cK\}\subset\cQ:=[-1,1]^m$. Define $\cF:\chi\rightarrow\Sigma_q$,
\begin{equation}
\label{eq:F}
\cF(x):=\{q\in\cQ:\ (x,q)\in\cK\}.
\end{equation}
First, we uniformly grid $\chi$ into $\bar{N}$ grid-points ($\bar{N}$ depending on the desired precision). Let $\{x^{(i)}\}_{i=1}^{\bar{N}}\subset\chi$ denote the points in the uniform grid. Next, for each grid point $x^{(i)}$, we sample from the distribution induced by the given finite Borel measure $\mu_q$ supported on $\cQ$. Let $\{q^{(i,k)}\}_{k=1}^{N_i}$ be $N_i$ i.i.d. sample of random parameter $q$. Then we approximate $\mu_q(\cF(x^{(i)}))$ by $$P^{(i)}_{N_i}:=\frac{1}{N_i}\sum_{k=1}^{N_i}\mathbf{1}_{\cK}\left(x^{(i)},q^{(i,k)}\right), \quad \hbox{where} \quad \mathbf{1}_{\cK}\left(x,q\right)=\left\{
                                       \begin{array}{ll}
                                         1, & \hbox{if $(x,q)\in\cK$;} \\
                                         0, & \hbox{otherwise.}
                                       \end{array}
                                     \right.$$
Because of law of large numbers, $\lim_{N_i\nearrow\infty}P^{(i)}_{N_i}=\mu_q(\cF(x^{(i)}))$. For each $x^{(i)}$, we chose sample size $N_i$ such that $P^{(i)}_{N_i}$ becomes stagnant to further increase in $N_i$. Finally, we approximate $x^*$ by $x^{(i^*)}$, where $i^*\in\argmax\{P^{(i)}_{N_i}: 1\leq i\leq \bar{N}\}$. It is clear that what we used is a \emph{naive} method, and it can be made much more efficient by using an adaptive gridding scheme on $\chi$. On the other hand, as the dimensions $n$ and $m$ are very small for the problems discussed in the numerical section, this naive method served its purpose.

\subsubsection{Example 1: A Simple Semialgebraic Set} Consider the chance optimization problem
\begin{equation} \label{max_exa2_def1}
\sup_{x\in\reals^5} \mu_q\left(\{q \in\reals^5 :\ \mathcal{P}(x,q)\geq0\ \}\right),
\vspace{-0.1cm}
\end{equation}
\vspace{-0.25cm}
where
\begin{equation*} 
\begin{array}{ll}
\mathcal{P}(x,q)=&0.185+ 0.5x_1 - 0.5x_2 + x_3 - x_4  + 0.5q_1 - 0.5q_2 + q_3 - q_4-x_1^2 - 2x_1q_1- x_2^2\\
 &- 2x_2q_2 - x_3^2 - 2x_3q_3 - x_4^2- 2x_4q_4 - x_5^2  + 2x_5q_5- q_1^2 - q_2^2 - q_3^2 - q_4^2 - q_5^2,  \end{array}
\end{equation*}
and the uncertain parameters $q_1,q_2,q_3,q_4,q_5$ have a uniform distribution: $q_1 \sim U[-1, 0]$, $q_2 \sim U[0, 1]$, $q_3 \sim U[-0.5, 1]$, $q_4 \sim U [-1, 0.5]$, $q_5 \sim U [0, 1]$ -- $U[a,b]$ denotes the uniform distribution between $a$ and $b$. The $k$-th moment of uniform distribution U[a,b] is $(\mathbf{y_q})_k =\frac{b^{k+1} - a^{k+1}}{(b-a)(k+1)} $. The optimum solution and corresponding optimal probability are obtained by Monte Carlo method: $x_1^* = 0.75$, $x_2^* = -0.75$, $x_3^* = 0.25$, $x_4^* = -0.25$, $x_5^* = 0.5$, and $P^* = 0.75 $. To obtain an approximate solution, we solve the SDP in \eqref{max_p4} using GloptiPoly and ALCC. For ALCC, we set $\nu_0$ to $1$, $5\times10^{-2}$ and $5\times10^{-3}$ when $d$ is equal to $1$, $2$, and $3$, respectively, and $\mathrm{tol}=1\times10^{-2}$. The results for relaxation order $d=1,2,3$ are shown in Table~\ref{tab:ex1}. As in Figure~\ref{fig:2}, when compared to $\mathbf{P}_d$, $\mathbf{\tilde{P}}_d$ approximates $\mathbf{P^*}$ better, i.e., when $\max\{\int\cP(x^d,q)~d\tilde{\mu}:\ \tilde{\mu}\preceq\mu_q,\ \tilde{\mu}\in \cM(\cF(x^d))\}$ is solved instead of $\max\{\int d\mu':\ \mu'\preceq\mu_q,\ \mu'\in \cM(\cF(x^d))\}$. 
We reported results up to order $d=3$, because for larger $d$, GloptiPoly did not terminate in 24 hours.
\begin{table}
\renewcommand{\arraystretch}{1.2}
\begin{footnotesize}
    \begin{center}
    \begin{tabular}{| l | l | l | l |}
    \hline
    \multicolumn{4}{|c|}{\textbf{ALCC}} \\ \thickhline
    $\mathbf{d}$ &  1 & 2 & 3 \\ \hline
    $\mathbf{n_{\mathrm{var}}}$ & 87 &  1127 &  8463  \\ \hline
    $\mathbf{iter}$& 169    & 624 & 1207 \\ \hline
    $\mathbf{cpu}$& 0.9  & 28.1 & 785.9\\ \hline
    $\mathbf{x_1}$ & 0.742  & 0.745  & 0.757 \\ \hline
    $\mathbf{x_2}$ & -0.777 & -0.701 &-0.721 \\ \hline
    $\mathbf{x_3}$ & 0.213  & 0.226  & 0.216  \\ \hline
    $\mathbf{x_4}$ & -0.239 & -0.250 & 0.236 \\ \hline
    $\mathbf{x_5}$ & 0.500  & 0.551  & 0.557 \\ \hline
    $\mathbf{P}_d$   & 0.991  & 0.971  & 0.961 \\ \thickhline
    $\mathbf{P}'_d$   & 1  & 1  & 1 \\ \hline
    $\mathbf{\tilde{P}}_d$   & 0.996  & 0.7739  & 0.6919 \\ \hline
    $\mathbf{\bar{P}}_d$   & 0.7504  & 0.7459  & 0.7459 \\ \thickhline

    \end{tabular}
    \begin{tabular}{| l | l | l | l |}
    \hline
\multicolumn{4}{|c|}{\textbf{GloptiPoly}} \\ \thickhline
    $\mathbf{d}$ & 1 & 2 & 3 \\ \hline
    $\mathbf{n_{\mathrm{var}}}$ & 87 &  1127 &  8463 \\ \hline
    $\mathbf{iter}$  & 18 & 25 & 41 \\ \hline
    $\mathbf{cpu}$  & 0.5 & 12.3 & 15324.3 \\ \hline
    $\mathbf{x_1}$ & 0.467 & 0.710 & 0.742 \\ \hline
    $\mathbf{x_2}$ & -0.467 & -0.710 &-0.742  \\ \hline
    $\mathbf{x_3}$ & 0.163 & 0.245 & 0.249  \\ \hline
    $\mathbf{x_4}$ & -0.163 & -0.245 & -0.249  \\ \hline
    $\mathbf{x_5}$ & 0.319 & 0.475 & 0.495\\ \hline
    $\mathbf{P}_d$ & 1 & 1 & 1 \\ \thickhline
    $\mathbf{P}'_d$ & 1 & 1 & 1 \\ \hline
    $\mathbf{\tilde{P}}_d$   & 0.9652  & 0.7768  & 0.7031 \\ \hline
    $\mathbf{\bar{P}}_d$   & 0.5067  & 0.7484  & 0.7535 \\ \thickhline

    \end{tabular}
    \end{center}
\end{footnotesize}
\vspace{3 mm}
\caption{ALCC and GloptiPoly results for Example 1}
\label{tab:ex1}
\end{table}

\subsubsection{Example 2: Union of Simple Sets} Given the following polynomials
\begin{align*} 
\mathcal{P}^{(1)}(x,q) = & - 0.263 + 0.4x_1 - 0.4x_2 + 0.8x_3 - 0.8x_4 + 1.2x_5 + 0.1q_1 + 0.08q_2 + 0.04q_3\\
&\hbox{}+ 0.4q_4 + 0.6q_5-x_1^2 - x_2^2 - x_3^2 - x_4^2 - x_5^2 - 0.5q_1^2 - 0.4q_2^2 - 0.1q_3^2 - q_4^2 - q_5^2,\\
\mathcal{P}^{(2)}(x,q) = &- 2.06+ 0.4x_1 - 0.8x_2 + 3.2x_3 - 1.6x_4 + 3.6x_5 - 0.4q_1 - 0.4q_2 - 0.2q_3  \\
&\hbox{}- 0.2q_4 - 0.8q_5 -x_1^2 - 2x_2^2 - 4x_3^2 - 2x_4^2 - 3x_5^2 - q_1^2 - q_2^2 - q_3^2 - q_4^2 - q_5^2,
\end{align*}
consider the chance optimization problem
\begin{equation} \label{max_exa3_def1}
\sup_{\rm x\in\reals^5} \mu_q \left(\bigcup_{j=1,2} \left\{ q\in\reals^5:\  \mathcal{P}^{(j)}(x,q)\geq0 \right\}\right),
\end{equation}
where $q_i \sim U[-0.5, 0.5]$ for all $i=1,\ldots,5$, i.e., the uncertain parameters $q_i$ are uniformly distributed on $[-0.5,0.5]$. The optimum solution and corresponding optimal probability are obtained by Monte Carlo method: $x_1^* = 0.2,~x_2^* = -0.2,~x_3^* = 0.4,~x_4^* = -0.4,~x_5^* = 0.6$, and $\mathbf{P^*} = 0.80 $. To obtain an approximate solution, we solve the SDP in \eqref{min_p3} using ALCC, where we set $\nu_0$ to $1$, $1\times10^{-1}$ and $1\times10^{-3}$ when $d$ is equal to $1$, $2$, and $3$, respectively, and $\mathrm{tol}=1\times10^{-2}$. The results for relaxation order $d=1,2,3$ are shown in Table~\ref{tab:alcc_ex2}. Let $\cF^{(k)}(x)=:\{q\in\cQ:\ \cP^{(k)}(x,q)\geq 0\}$ for $k=1,2$. The probability estimates $\mathbf{\tilde{P}}_d$ reported in Table~\ref{tab:alcc_ex2} are computed by solving the SDP relaxation for
\begin{equation*}
\max\left\{\int\cP^{(1)}(x^d,q)~d\tilde{\mu}_1+\int\cP^{(2)}(x^d,q)~d\tilde{\mu}_2:\ \tilde{\mu}_1+\tilde{\mu}_2\preceq\mu_q,~\tilde{\mu}_1\in \cM(\cF^{(1)}(x^d)),~\tilde{\mu}_2\in \cM(\cF^{(2)}(x^d))\right\}.
\end{equation*}
For this example, GloptiPoly fails to extract the optimum solution.

\begin{table}
\renewcommand{\arraystretch}{1.2}
\begin{footnotesize}
    \begin{center}
   \begin{tabular}{| l | l | l | l |}
    \hline
    \multicolumn{4}{|c|}{\textbf{ALCC}} \\ \thickhline
    $\mathbf{d}$ &  1 & 2 & 3 \\ \hline
    $\mathbf{n_{\mathrm{var}}}$ & 153 &   2128 &  16478 \\ \hline
    $\mathbf{iter}$&  979  & 1467 & 1875 \\ \hline
    $\mathbf{cpu}$& 6.5  &  102.2 & 434.7\\ \hline
    $\mathbf{x_1}$ &  0.209 & 0.328  & 0.201 \\ \hline
    $\mathbf{x_2}$ & -0.202 & -0.174 & -0.201 \\ \hline
    $\mathbf{x_3}$ & 0.397  & 0.466  & 0.430  \\ \hline
    $\mathbf{x_4}$ & -0.400 & -0.405 & -0.401 \\ \hline
    $\mathbf{x_5}$ & 0.667  & 0.638  & 0.591 \\ \hline
    $\mathbf{P}_d$   & 1  & 0.997  & 0.981 \\ \thickhline
    $\mathbf{P}'_d$   & 1  & 1 & 1 \\ \hline
    $\mathbf{\tilde{P}}_d$   & 0.9973  & 0.8610  & 0.8926 \\ \hline
    $\mathbf{\bar{P}}_d$   & 0.8937  & 0.8745  & 0.8984 \\ \thickhline
    \end{tabular}
    \end{center}
\end{footnotesize}
\vspace{3 mm}
\caption{ALCC results for Example 2}
\vspace{-5 mm}
\label{tab:alcc_ex2}
\end{table}
%
%
\subsubsection{Example 3: Portfolio Selection Problem} We aim at selecting a portfolio of financial assets to maximize the probability of achieving a return higher than a specified amount $r^*$. Suppose that for each asset $i=1,...,N$, its uncertain rate of return is a random variable $\xi_i(q)$; and let $(\cQ,\Sigma_q,\mu_q)$ denote the underlying probability space. In this context $x_i$ denotes the percentage of money invested in asset $i$. 
More precisely, we solve the following problem:
\begin{equation} \label{eq:ex3}
\sup_{x\in\reals^N} \mu_q\left(\left\{q \in\reals^N :\ \sum^N_{i=1}\xi_i(q)x_i\geq r^*\right\}\right)\quad \hbox{s.t.}\quad \sum_{i=1}^Nx_i\leq 1,\quad x_i\geq 0\,\quad \forall\ i\in\{1,\ldots,N\}.
\end{equation}
In our example problem, $r^*=1.5$, $N=4$, $\xi_1(q)= 1+q_1$, $\xi_2(q)= 1+q_2$, $\xi_3(q)= 0.9+q_3$, $\xi_4(q)= 0.9+q_4$, where $\{q_i\}_{i=1}^4$ are independent, and $q_1 \sim \mathrm{Beta}(3-\sqrt 2,3+\sqrt 2)$, $q_2 \sim \mathrm{Beta}(4,4)$, $q_3 \sim \mathrm{Beta}(3+\sqrt 2,3-\sqrt 2)$, $q_4 \sim U [0.5, 1]$. The $k$-th moment of Beta distribution $\mathrm{Beta}(\alpha,\beta)$ over [0,1] is $y_k = \frac{\alpha+k-1}{(\alpha+\beta+k-1)} y_{k-1}$ and $y_0=1$. We will solve an equivalent problem in the form of \eqref{intro_max} with $\ell=7$, where $\cP_j(x,q)=x_j$ for $j=1,\ldots,4$, $\cP_5(x,q)=1 - \sum_{i=1}^4x_i$, $\cP_6(x,q)=8-\sum_{i=1}^4x_i^2-\sum_{i=1}^4q_i^2$, and $\cP_7(x,q)=\sum_{i=1}^4 \xi_i(q) x_i-r^*$.
Since any $(x,q)\in\cK$ satisfies $x\in\chi$ and $q\in\cQ$, we added polynomial $\mathcal{P}_6(x,q)$ to assure that the resulting representation of the semialgebraic set $\cK$ satisfies Putinar's property. The optimum solution and the corresponding optimal probability are computed approximately by Monte Carlo method: $x_1^* = 0$, $x_2^* = 0$, $x_3^* = 0.3$, $x_4^* = 0.7$, and $P^* = 0.89 $. To obtain an approximate solution, we solve the SDP relaxation in \eqref{max_p4} using GloptiPoly and ALCC. For ALCC, we set $\nu_0$ to $1\times10^{-2}$, $1\times10^{-2}$ and $1\times10^{-3}$ when $d$ is equal to $1$, $2$, and $3$, respectively, and $\mathrm{tol}=1\times10^{-3}$. The results for relaxation order $d=1,2,3$ are shown in Table~\ref{tab:ex3}. We reported results up to order $d=3$, because for larger $d$, GloptiPoly did not terminate in 24 hours.
\begin{table}[!h]
\renewcommand{\arraystretch}{1.2}
\begin{footnotesize}
    \begin{center}
    \begin{tabular}{| l | l | l | l | l |}
    \hline
\multicolumn{4}{|c|}{\textbf{ALCC}} \\ \thickhline
    $\mathbf{d}$ & 1 & 2 & 3  \\ \hline
    $\mathbf{n_{\mathrm{var}}}$ & 60 &  565 &  3213  \\ \hline
    $\mathbf{iter}$ & 573 &  388 &  2227  \\ \hline
   $\mathbf{cpu}$ & 3.625 &  16.426 &   756.798 \\ \hline
    $\mathbf{x_1}$ & 0.004 & 0.009 & 0.002 \\ \hline
    $\mathbf{x_2}$ &  0.012 & 0.009 & 0.006  \\ \hline
    $\mathbf{x_3}$ &  0.438 & 0.449 & 0.299 \\ \hline
    $\mathbf{x_4}$ & 0.5007 & 0.522 & 0.677  \\ \hline
    $\mathbf{P}_d$ & 0.996 & 0.994 & 0.980 \\ \thickhline
    $\mathbf{P}'_d$ & 1 & 1 & 0.9716 \\ \hline
    $\mathbf{\tilde{P}}_d$ & 0.7928 & 0.8177 & 0.8220 \\ \hline
    $\mathbf{\bar{P}}_d$ & 0.7405 & 0.8655 & 0.8422 \\ \thickhline
    \end{tabular}
    \begin{tabular}{| l | l | l | l |}
    \hline
\multicolumn{4}{|c|}{\textbf{GloptiPoly}} \\ \thickhline
    $\mathbf{d}$ & 1 & 2 & 3  \\ \hline
    $\mathbf{n_{\mathrm{var}}}$ & 60 &  565 &  3213 \\ \hline
    $\mathbf{iter}$ & 15 & 20  & 48  \\ \hline
    $\mathbf{cpu}$ &  0.509 & 2.617 & 1025.045 \\ \hline
    $\mathbf{x_1}$ & 0.133 & 0.0462 & 0.003 \\ \hline
    $\mathbf{x_2}$ & 0.192 & 0.154 & 0.075  \\ \hline
    $\mathbf{x_3}$ & 0.295 & 0.297 & 0.210 \\ \hline
    $\mathbf{x_4}$ & 0.325 & 0.493 & 0.710 \\ \hline
    $\mathbf{P}_d$ & 1 & 1 & 0.999 \\ \thickhline
    $\mathbf{P}'_d$ & 0.9071 & 0.9997 & 0.9896 \\ \hline
    $\mathbf{\tilde{P}}_d$ & 0.3808 & 0.7753 & 0.8395 \\ \hline
    $\mathbf{\bar{P}}_d$ & 0.3865 & 0.8267 & 0.8675 \\ \thickhline
    \end{tabular}
    \end{center}
\end{footnotesize}
\vspace{3 mm}
\caption{ALCC and GloptiPoly results for Example 3}
\vspace{-3 mm}
\label{tab:ex3}
\end{table}
\subsubsection{Example 4: Nonlinear Control Problem} In this example, we consider the controller design problem for the following uncertain nonlinear dynamical system. For a given control parameter vector $K\in\reals^3$, let the system $x(k)^T=[x_1(k),~x_2(k),~x_3(k)]\in\reals^3$ satisfy
\begin{equation}
\label{eq:control_system}
\begin{array}{r l}
u(k) =&K_1x_1(k)+K_2x_2(k)+K_3x_3(k),\\
x_1(k+1)=&\Delta~x_2(k),\\
x_2(k+1)=&x_1(k)~x_3(k),\\
x_3(k+1)=&1.2~x_1(k)-0.5~x_2(k)+x_3(k)+u(k),
\end{array}
\end{equation}
for $k=0,1$, where $x_{1}(0) \sim U[-1, 1]$,~ $x_{2}(0) \sim U[-1, 1]$,~ $x_{3}(0) \sim U[-1, 1]$,~ $\Delta \sim U [-0.4, 0.4]$, i.e., initial state vector $x(0)$, and model parameter $\Delta$ are uncertain and uniformly distributed. The objective is to lead the system using state feedback control $u(k)$ to the cube centered at the origin with the edge length of $0.2$ in at most $2$ steps by properly choosing the control decision variables $\{K_i\}_{i=1}^3$ such that $-1 \leq K_i \leq 1$. The equivalent chance problem is stated in \eqref{eq:controller_problem}, where $\mathbf{e}^T=[1,1,1]$.
\begin{align} \label{eq:controller_problem}
\sup_{K\in\reals^3} &\ \mu_q\left(\left\{\Big(x(0),\Delta\Big):\ -0.1\mathbf{e}\leq x(2)\leq 0.1\mathbf{e} \right\}\right),\\
\hbox{s.t.} &\ \{x(k),u(k)\}_{k=0}^2 \hbox{ satisfy } \eqref{eq:control_system}, \nonumber\\
&\ -\mathbf{e} \leq K\leq \mathbf{e}. \nonumber
\end{align}
The following optimal solution and the corresponding optimal probability are computed by Monte Carlo method: $K_1^* = -1$, $K_2^* = 0.5$, $K_3^* = -0.9$, and $\mathbf{P^*} = 0.84$.  To obtain an equivalent SDP formulation for the chance constrained problem in \eqref{eq:controller_problem}, $x(2)$ is explicitly written in terms of control vector $K\in\reals^3$ and uncertain parameters, $x(0)$ and $\Delta$, using the dynamic system given in \eqref{eq:control_system}:
\begin{align*} 
x_1(2) =~&\Delta~x_{1}(0)x_{3}(0),\\
x_2(2) =~&(1.2+K_1)\Delta~x_{1}(0)x_{2}(0) +(K_2-0.5)\Delta~x_{2}(0)^2 + (1+K_3)\Delta~x_{2}(0)x_{3}(0),\\
x_3(2) =~&(1+2K_3+K_3^2)~x_{3}(0)+(K_2-0.5K_3-0.5+1.2\Delta +K_1\Delta +K_2K_3)~x_{2}(0)\\
&\hbox{}+(1.2+K_1+1.2K_3+K_1K_3)~x_{1}(0)+(K_2-0.5)~x_{1}(0)x_{3}(0).
\end{align*}
%
Based on the obtained polynomials, the minimum relaxation order for this problem is 2. To obtain an approximate solution, we solve the SDP in \eqref{max_p4} using GloptiPoly and ALCC. For ALCC, we set $\nu_0$ to $5\times10^{-3}$, $5\times10^{-3}$ and $1\times10^{-3}$ when $d$ is equal to $2$, $3$ and $4$, respectively, and $\mathrm{tol}=1\times10^{-3}$. The results for relaxation order $d=2,3,4$ are shown in Table~\ref{tab:ex4}. 
\begin{table}
\renewcommand{\arraystretch}{1.2}
\begin{footnotesize}
    \begin{center}
    \begin{tabular}{| l | l | l | l |}
    \hline
\multicolumn{4}{|c|}{\textbf{ALCC}} \\ \thickhline
    $\mathbf{d}$ & 2 & 3 & 4  \\ \hline
    $\mathbf{n_{\mathrm{var}}}$ & 365 &   1800 &   6600  \\ \hline
    $\mathbf{iter}$ & 416 & 4300 & 5325 \\ \hline
    $\mathbf{cpu}$ &  14.934 &   897.708 &  5318.387  \\ \hline
    $\mathbf{K_1}$ & 0 & -0.244 & -0.683 \\ \hline
    $\mathbf{K_2}$ & 0 & 0.468 & 0.476  \\ \hline
    $\mathbf{K_3}$ & 0 & -0.868 & -0.868  \\ \hline
    $\mathbf{P}_d$ & 0.238 & 0.996 & 0.983  \\ \thickhline
    $\mathbf{P}'_d$ & 0.65 & 0.9 & 0.982  \\ \hline
    $\mathbf{\bar{P}}_d$ & 0.061 & 0.445 & 0.685  \\ \thickhline
    \end{tabular}
    \begin{tabular}{| l | l | l | l |}
    \hline
    \multicolumn{4}{|c|}{\textbf{GloptiPoly}} \\ \thickhline
    $\mathbf{d}$ & 2 & 3 & 4  \\ \hline
    $\mathbf{n_{\mathrm{var}}}$ & 365 &   1800 &   6600 \\ \hline
    $\mathbf{iter}$ & 19 & 26 & 36  \\ \hline
    $\mathbf{cpu}$ &  1.3 &   99.2 &  10389.8\\ \hline
    $\mathbf{K_1}$ & 0 & -0.492 & -0.796 \\ \hline
    $\mathbf{K_2}$ & 0 & 0.439 & 0.487 \\ \hline
    $\mathbf{K_3}$ & 0 & -0.823 & -0.891  \\ \hline
    $\mathbf{P}_d$ & 1 & 1 & 1  \\ \thickhline
    $\mathbf{P}'_d$ & 0.65 & 0.959 & 0.999  \\ \hline
    $\mathbf{\bar{P}}_d$ & 0.061 & 0.508 & 0.766  \\ \thickhline
    \end{tabular}
    \end{center}
  \end{footnotesize}
  \vspace{3 mm}
  \caption{ALCC and GloptiPoly results for Example 4}
  \vspace{-5 mm}
  \label{tab:ex4}
\end{table}
\subsubsection{Example 5: Run time} In this example, for fixed degree of the relaxation order d, we examined how the run times of ALCC algorithm scale as the problem size increases. For this purpose, we consider the following problem: Given $n\geq 1$,
we set $\cP:\reals^n\times\reals^n\rightarrow\reals$, $\mathcal{P}\left(x,q\right)=0.81-\sum_{i=1}^n (x_i-q_i)^2$; and solve
\begin{equation} \label{max_exa1_def10}
\sup_{x\in R^n} \mu_q\left(\{q \in R^n :\ \mathcal{P}(x,q)\geq0\ \}\right).
\end{equation}
The numerical results for increasing $n$ and fixed relaxation order $d=1$ are displayed in Table~\ref{tab:ex5}. For each $n$, ALCC recovered the optimal decision value: $x^*=0$. 
\begin{table}[h!]
\begin{footnotesize}
    \begin{center}
    \begin{tabular}{| l | l | l | l | l | l | l | l | l | l |}
    \hline
    \multicolumn{10}{|c|}{\textbf{ALCC}} \\ \hline
    $\mathbf{n}$ & 5 &  10 &  20 & 30 & 40 & 50 & 60 & 70 & 80  \\ \hline
    $\mathbf{d}$ &  1 & 1 & 1 &  1 & 1 & 1 &  1 & 1 & 1 \\ \hline
    $\mathbf{n_{\mathrm{var}}}$ & 10 &  20 &  40 & 60 & 80 & 100 & 120 & 140 & 160  \\ \hline
    $\mathbf{iter}$& 82    & 140 & 97 & 182 & 201 & 175 & 191 & 186 & 208 \\ \hline
    $\mathbf{cpu}$& 0.3969  & 1.5349 & 3.5542 & 14.2899 & 27.7978 & 37.2624 & 60.4454 & 83.3669 & 122.7844\\ \hline
    \end{tabular}
    \end{center}
\end{footnotesize}
  \caption{ALCC for increasing problem in Example 5}
  \label{tab:ex5}
\end{table}
\section{ Conclusion}
In this paper, ``chance optimization" problems are introduced, where one aims at maximizing the probability of a set defined by polynomial inequalities. These problems are, in general, nonconvex and computationally hard. A sequence of semidefinite relaxations is provided whose sequence of optimal values is shown to converge to the optimal value of the original problem. To solve the semidefinite programs of increasing size obtained by relaxing the original chance optimization problem, a first-order augmented Lagrangian algorithm is implemented which enables us to solve much larger size semidefinite programs that interior point methods can deal with. Numerical examples are provided that show that one can obtains reasonable approximations to the optimal solution and the corresponding optimal probability even for lower order relaxations.

\appendix
\section{Sample GloptiPoly Code for Chance Optimization}
In this section, we provide the Gloptipoly code for solving the simple problem given in \eqref{max_exa1_def1} and \eqref{max_exa1_def2} of Section~\ref{sec:simple_example}.
\begin{figure}[!h]
{\footnotesize
\texttt{>> d=2; \%relaxation order}\\
\texttt{>> \%  mu\textunderscore s: slack measure,  mu\textunderscore s = mux muq - mu, y\textunderscore s: moments of mu\textunderscore s}\\
\texttt{>> mpol x\textunderscore s q\textunderscore s;  mu\textunderscore s = meas([x\textunderscore s;q\textunderscore s]);  y\textunderscore s=mom(mmon([x\textunderscore s;q\textunderscore s],2*d));  }\\
\texttt{>> \% mu: measure supported on p>=0, y: moments of mu}\\
\texttt{>> mpol x q; mu = meas([x;q]); y=mom(mmon([x;q],2*d)); }\\
\texttt{>> p=0.5*q*(q\textasciicircum 2+(x-0.5)\textasciicircum 2)-(q\textasciicircum4+q\textasciicircum2*(x-0.5)\textasciicircum2+(x-0.5)\textasciicircum4);}\\
\texttt{>> \%  mux: measure, yx: moments of mux}\\
\texttt{>> mpol xm; mux= meas([xm]); yx=mom(mmon([xm],2*d)); }\\
\texttt{>> \%  yq: moments of uniform distribution muq on [-1,1] }\\
\texttt{>> yq=[1;0;1/3;0;0.2];}\\
\texttt{>> \% yxq : moments of upper bound measure mux muq}\\
\texttt{>> yxq = [yx(1)*yq(1);yx(2)*yq(1);yx(1)*yq(2);yx(3)*yq(1);yx(2)*yq(2);}\\
\texttt{>> yx(1)*yq(3);yx(4)*yq(1);yx(3)*yq(2);yx(2)*yq(3);yx(1)*yq(4);}\\
\texttt{>>   yx(5)*yq(1);yx(4)*yq(2);yx(3)*yq(3);yx(2)*yq(4);yx(1)*yq(5)];}\\
\texttt{>>   Pd=msdp(max(mass(mu)),mass(mux)==1,p>=0,y\textunderscore s==yxq - y,-1<=yx,yx<=1,d);msol(Pd);}\\
\texttt{>>  y=double(mvec(mu)); yx=double(mvec(mux)); \% results}\\
\texttt{>>  Decision= yx(2)}\\
\texttt{>>  Probability = y(1)}
}
\caption{GloptiPoly Code in Matlab for Example 1}
\label{fig:3}       
\end{figure}
\end{document}